\documentclass[a4paper, 12pt]{amsart}
\usepackage{a4wide}
\usepackage{amsmath,amsthm,amssymb}
\usepackage[arrow,matrix]{xy}
\usepackage{hyperref}
\usepackage{color}
\usepackage{pstricks,pst-node}
\hypersetup{
    pdftitle=   {Topological classification of quasitoric manifolds with the second Betti number 2},
    pdfauthor=  {Suyoung Choi, Seonjeong Park and Dong Youp Suh}
}
\theoremstyle{ams}
\newtheorem{theorem}{Theorem}[section]
\newtheorem{proposition}[theorem]{Proposition}
\newtheorem{lemma}[theorem]{Lemma}
\newtheorem{corollary}[theorem]{Corollary}
\numberwithin{equation}{section}

\theoremstyle{definition}

\newtheorem{remark}[theorem]{Remark}
\newtheorem{example}[theorem]{Example}

\newcommand{\C}{\mathbb{C}}
\newcommand{\Q}{\mathbb{Q}}

\newcommand{\Z}{\mathbb{Z}}
\newcommand{\N}{\mathbb{N}}
\newcommand{\CP}{\mathbb{C}P}
\newcommand{\GL}{\mathrm{GL}}

\newcommand{\vc}{\mathrm{vc}}
\newcommand{\cF}{\mathcal{F}}
\newcommand{\va}{\mathbf{a}}
\newcommand{\bb}{\mathbf{b}}
\newcommand{\uu}{\mathbf{u}}
\newcommand{\vv}{\mathbf{v}}
\newcommand{\vs}{\mathbf{s}}
\newcommand{\vr}{\mathbf{r}}
\newcommand{\vw}{\mathbf{w}}
\newcommand{\vz}{\mathbf{z}}
\newcommand{\betti}{\beta}
\DeclareMathOperator{\aut}{Aut}
\DeclareMathOperator{\Hom}{Hom}

\begin{document}
\title[Classification of quasitoric manifolds with $\betti_2=2$]{Topological classification of quasitoric manifolds with the second Betti number $2$}

\author[S.Choi]{Suyoung Choi}
\address{Department of Mathematics, Ajou University, San 5, Woncheon-dong, Yeongtong-gu, Suwon 443-749, Korea}
\email{schoi@ajou.ac.kr}

\author[S.Park]{Seonjeong Park}
\address{Department of Mathematical Sciences, KAIST, 335 Gwahangno, Yu-sung Gu, Daejeon 305-701, Korea}
\email{psjeong@kaist.ac.kr}

\author[D.Y.Suh]{Dong Youp Suh}
\address{Department of Mathematical Sciences, KAIST, 335 Gwahangno, Yu-sung Gu, Daejeon 305-701, Korea}
\email{dysuh@math.kaist.ac.kr}

\thanks{The first author is supported by the new faculty research fund by Ajou University. The second author is supported by the second stage of the Brain Korea 21 Project, the Development of Project of Human Resources in Mathematics, KAIST in 2011. The third author is partially supported by Basic Science Research Program through the national Research Foundation of Korea(NRF) founded by the Ministry of Education, Science and Technology (2010-0001651).}


\date{\today}
\maketitle

\begin{abstract}
    A quasitoric manifold is a $2n$-dimensional compact smooth manifold with a locally standard action of an $n$-dimensional torus whose orbit space is a simple polytope. In this article, we classify quasitoric manifolds with the second Betti number $\betti_2=2$ topologically. Interestingly, they are distinguished by their cohomology rings up to homeomorphism.
\end{abstract}

\section{Introduction}
    The notion of a quasitoric manifold was introduced by Davis and Januszkiewicz~\cite{DJ}. A \emph{quasitoric manifold} $M$ is a $2n$-dimensional compact smooth manifold with a locally standard action of an $n$-dimensional torus $T^n=(S^1)^n$, whose orbit space can be identified with an $n$-dimensional simple polytope $P$. Here, the orbit map $\pi \colon M \to P$ maps every $k$-dimensional orbit to a point in the interior of a codimension-$k$ face of $P$ for $k=0,\ldots,n$. A typical example of a quasitoric manifold is
    a complex projective space $\CP^n$
    of complex dimension $n$
    with the standard $T^n$-action whose orbit space is the $n$-simplex $\Delta^n$.

    A quasitoric manifold is a topological analogue of a non-singular projective toric variety. A \emph{toric variety} $X$ of complex dimension $n$ is a normal algebraic variety which admits an action of an algebraic torus $(\C^\ast)^n$ having a dense orbit. We call a non-singular compact toric variety a \emph{toric manifold}. Note that we have the restricted action of $T^n = (S^1)^n \subset (\C^\ast)^n$ on a toric manifold $X$. One can easily show that this action is locally standard, and if $X$ is projective, then there is a moment map whose image is a simple convex polytope. Hence, all projective toric manifolds are quasitoric manifolds. However, the converse is not always true. For instance, $\CP^2 \sharp \CP^2$ with an appropriate $T^2$-action is a quasitoric manifold over $\Delta^1 \times \Delta^1$ but not a toric manifold, because there is no
    almost complex structure on $\CP^2\#\CP^2$.
    Therefore, the notion of a quasitoric manifold can be regarded as a topological generalization of that of a projective toric manifold in algebraic geometry.

    We shall investigate quasitoric manifolds $M$ with the second Betti number $\betti_2=2$. As will be remarked in Section~\ref{sec:quasitoric manifolds of $b_2=2$}, the orbit space of $M$ can be identified with a product of two simplices. The classification of projective toric manifolds with $\betti_2=2$ as varieties
    was completed by Kleinschmidt \cite{K}. More generally, toric manifolds over a product of simplices were studied by Dobrinskaya \cite{D} and Choi-Masuda-Suh \cite{ch-ma-su08}. These toric manifolds are known as generalized Bott manifolds. In particular, toric manifolds with $\betti_2=2$ are two-stage generalized Bott manifolds, which will be explained in Section~\ref{sec:quasitoric manifolds of $b_2=2$}. It is shown in \cite{CMS} that all two-stage generalized Bott manifolds are classified by their cohomology rings, which gives the smooth classification of toric manifolds with $\betti_2=2$.

    The purpose of this paper is to classify quasitoric manifolds with $\betti_2=2$ up to homeomorphism.
    For this, we show that if the cohomology ring of a quasitoric manifold is isomorphic to that of a two-stage generalized Bott manifold, then the quasitoric manifold is homeomorphic to a two-stage generalized Bott manifold. We also show that for a polytope which is the product of two simplices there are only finitely many quasitoric manifolds over the polytope, which are not homeomorphic to generalized Bott manifolds. As we will see in the paragraph after~\eqref{characteristic matrix}
    on page~\pageref{M}, any quasitoric manifold with $\betti_2=2$ can be written as $M_{\va,\bb}$ for some $\va \in \Z^m$ and $\bb \in \Z^n$, where the orbit space of $M_{\va,\bb}$ is $\Delta^n \times \Delta^m$. Then we have the following topological classification.

    \begin{theorem}\label{classify}
        Any quasitoric manifold with the second Betti number $\betti_2=2$
        is homeomorphic to either a two-stage generalized Bott manifold, or
$$
    M_{\vs,\vr}  \quad\text{ for $\vs:=(2,\ldots,2,0,\ldots,0) \in \Z^m$ and $\vr:=(1,\ldots,1,0,\ldots,0) \in \Z^n$},
$$
        where the number of nonzero components in $\vs$, respectively $\vr$, is less than or equal to $\lfloor \frac{m+1}{2} \rfloor$, respectively $\lfloor \frac{n+1}{2} \rfloor$.
        Moreover, if $n$ or $m$ is $1$, then $M_{\vs,\vr}$ is a two-stage generalized Bott manifold, or $\CP^{m+n} \# \CP^{m+n}$,
        or $M_{2,(1,0,\ldots,0)}$.
    \end{theorem}
    More precise classification results are summarized in Section~\ref{sec:Classification of quasitoric manifolds}. Note that there is an interesting quasitoric manifold over $\Delta^n \times \Delta^1$ which is homeomorphic to a generalized Bott manifold, but has no $T^{n+1}$-invariant almost complex structure; namely, $M_{2,(1,0)}$ is such a quasitoric manifold that is homeomorphic to a generalized Bott manifold $M_{2,(0,0)}$,
    as we will see in Lemma~\ref{lem:class a=pm2}.

    Furthermore, we can show that $M_{\va,\bb}$ and $M_{\va',\bb'}$
    with $M_{\va,\bb}/T$ and $M_{\va',\bb'}/T$ combinatorially equivalent to
    $\Delta^n \times \Delta^m$ are homeomorphic if and only if their cohomology rings are isomorphic as graded rings. In addition, the combinatorial types of certain polytopes are completely determined by the cohomology rings of quasitoric manifolds over those polytopes, see~\cite{CPS}. Products of simplices belong to the class of polytopes that have this property.
    That is, for a quasitoric manifold $M$, if the cohomology ring of $M$ is isomorphic to that of $M_{\va,\bb}$, then the orbit space of $M$ is combinatorially equivalent to the orbit space of $M_{\va,\bb}$.

    As a consequence, we have the following main theorem of this paper, which does not include any assumption on the type of the base polytope:
    \begin{theorem}\label{main}
        Two quasitoric manifolds with $\betti_2=2$ are homeomorphic if and only if their cohomology rings are isomorphic as graded rings.
    \end{theorem}

    This research is motivated by the \emph{cohomological rigidity problem} for quasitoric manifolds which asks whether the homeomorphism types of quasitoric manifolds are distinguished by their cohomology rings or not, see~\cite{MS} for the problem and other related problems. In general, the cohomological rigidity problem is rather bold because the cohomology ring as an invariant is not sufficient to determine topological types of manifolds. Indeed, many classical results such as \cite{H} provide many examples of pairs of manifolds which are homotopic but not homeomorphic. However, many $2n$-dimensional manifolds do not have $T^n$-symmetry, and, so far, there is no counterexample for the cohomological rigidity problem. On the contrary, some affirmative partial evidence is given by recent papers such as \cite{MP}, \cite{CMS}, \cite{CS}, \cite{CM} and others. Theorem~\ref{main} also gives another affirmative partial answer to the rigidity problem. For more information about rigidity problem, we refer the reader to the survey paper~\cite{ch-ma-su11}.

    This paper is organized as follows. In Section~\ref{Preliminaries}, we recall general facts on quasitoric manifolds and moment angle manifolds.
    In Section~\ref{sec:quasitoric manifolds of $b_2=2$}, we introduce generalized Bott manifolds, and deal with the cohomology rings of quasitoric manifolds with $\betti_2=2$.
    We find a necessary and sufficient condition for a quasitoric manifold to be equivalent to a generalized Bott manifold in some specific cases in Section~\ref{sec:quasitoric mfds equiv to GBM}.
    In Sections~\ref{sec:classification of quasitoric manifolds when m=1} and \ref{sec:quasi mfds over D^m x D^n}, we prove Theorem~\ref{classify}, and prepare to prove Theorem~\ref{main} by classifying quasitoric manifolds $M_{a,\bb}$ and $M_{\vs,\vr}$ up to homeomorphism. In Section~\ref{sec:proof of main}, we give a full proof of Theorem~\ref{main}. In the final section, we give the complete topological classification of quasitoric manifolds with $\betti_2=2$.

\section{Preliminaries}\label{Preliminaries}
    An $n$-dimensional (combinatorial) polytope is called \emph{simple} if exactly $n$ facets (codimension-one face) meet
    at each vertex. Let $P$ be a simple polytope of dimension-$n$ with $d$ facets, and let $\cF(P)=\{F_1,\ldots,F_d\}$ be the set of facets of $P$. Now consider a map $\lambda \colon \cF(P) \rightarrow \Z^n$ which satisfies the following \emph{non-singularity condition};
    \begin{equation}\label{non-singularity}
    \begin{array}{l}
    \lambda(F_{i_1}), \ldots, \lambda(F_{i_{\alpha}}) \mbox{ form a part of an integral basis of } \Z^n\\
    \mbox{whenever the intersection } F_{i_1} \cap \cdots \cap F_{i_{\alpha}} \mbox{ is non-empty.}
    \end{array}
    \end{equation}
    Such $\lambda$ is called a \emph{characteristic function}, and $\lambda(F_i)$ is called a \emph{facet vector} of $F_i$. For a characteristic function $\lambda \colon \cF(P) \rightarrow \Z^n$ and a face $F$ of $P$, we denote by $T(F)$ the subgroup of $T^n$ corresponding to the unimodular subspace of $\Z^n$ spanned by $\lambda(F_{i_1}), \ldots, \lambda(F_{i_\alpha})$, where $F=F_{i_1} \cap \cdots \cap F_{i_\alpha}$.

    Given a characteristic function $\lambda$ on $P$, we construct a manifold
    \begin{equation}\label{construction:quasitoric}
        M(\lambda) := T^n \times P /\sim,
    \end{equation}
    where $(t,p) \sim (s,q)$ if and only if $p=q$ and $t^{-1}s \in T(F(p))$, where $F(p)$ is
    the face of $P$ which contains $p \in P$ in its relative interior.
    The standard $T^n$-action on $T^n$ induces a free action of $T^n$ on $T^n \times P$, which descends to an effective
    action on $M(\lambda)$ whose orbit space is $P$. Since this action is locally standard, $M(\lambda)$ is indeed a quasitoric manifold over $P$.

    Two quasitoric manifolds $M_1$ and $M_2$ over $P$ are said to be \emph{equivalent} if there is a $\theta$-equivariant homeomorphism $f\colon M_1 \rightarrow M_2$, i.e. $f(gm)=\theta(g)f(m)$ for $g \in T^n$ and $m \in M_1$, which covers the identity map on $P$ for some automorphism $\theta$ of $T^n$. It is obvious from the definition of the equivalence that $M(\lambda_1)$ and $M(\lambda_2)$ are equivalent if there is an automorphism $\sigma \in \aut(\Z^n)=\GL(\Z, n)$ such that $\lambda_1 = \sigma \circ \lambda_2$. By Davis and Januszkiewicz \cite{DJ}, every quasitoric manifold is represented by a pair of $P$ and $\lambda$ up to equivalence.

    Note that one may assign an $n \times d$ matrix $\Lambda$ to a characteristic function $\lambda$ by
    $$
        \Lambda = ( \lambda(F_1) \cdots \lambda(F_d) ) = (A|B),
    $$
    where $A$ is an $n \times n$ matrix and $B$ is an $n \times (d-n)$ matrix. We call $\Lambda$ a \emph{characteristic matrix}. By additionally setting $F_1 \cap \cdots \cap F_n \neq \emptyset$, we may assume that the matrix $A=(\lambda(F_1), \ldots, \lambda(F_n))$ is invertible from the nonsingularity condition \eqref{non-singularity}. Moreover, the inverse $A^{-1}$ belongs to $\GL(\Z,n)$. Thus, up to equivalence, the corresponding matrix $\Lambda$ can be represented by $(E_n | A^{-1}B)$, where $E_n$ is the identity matrix of size~$n$.

    \begin{remark}\label{rmk:sign}
        Let $\Lambda$ be the above characteristic matrix corresponding to a quasitoric manifold $M$. If we let
        $$D_{k,n}:=\mathrm{diag}(1,\ldots,1,-1,1,\ldots,1)$$
        be the diagonal $n \times n$ matrix whose $k$-th diagonal entry is $-1$ and the others are $1$, then the matrix $D_{k,n}\Lambda D_{\ell,d}$ is the matrix obtained from $\Lambda$ by changing the signs of $k$-th row and $\ell$-th column, where $1 \leq k \leq n$ and $1 \leq \ell \leq d$.
        Since two vectors $\lambda(F_i)$ and $-\lambda(F_i)$ determine the same circle subgroup of $T^n$, the sign of a facet vector does not affect the corresponding quasitoric manifold from the construction~\eqref{construction:quasitoric}. Thus $\Lambda D_{\ell,d}$ is still a characteristic matrix corresponding to $M$. Hence $D_{k,n}\Lambda D_{\ell,d}$ can also be a characteristic matrix corresponding to $M$, up to equivalence, because $D_{k,n} \in \GL(\Z,n)$.
    \end{remark}

    Let $\Z[v_1, \ldots, v_d]$ denote the polynomial ring in $d$ variables over $\Z$ with $\deg v_i = 2$. We identify each $F_i \in \cF(P)$ with the indeterminate $v_i$. The \emph{face ring} (or \emph{Stanley-Reisner ring}) $\Z(P)$ of $P$ is the quotient ring $$\Z(P) =\Z[v_1, \ldots, v_d]/I_P,$$
    where $I_P$ is the ideal generated by the monomials $v_{i_1}\cdots v_{i_\ell}$
    with $F_{i_1}\cap\cdots\cap F_{i_\ell}=\emptyset$.

    Let $M$ be a quasitoric manifold over $P$ with projection $\pi\colon M \to P$ and the characteristic function $\lambda$. Then one can find an isomorphism between $\Z(P)$ and the equivariant cohomology ring of $M$ with $\Z$ coefficients:
    $$
     H_T^\ast (M)\cong \Z[v_1, \ldots, v_d]/I_P = \Z(P),
    $$
    where $v_j$ is the equivariant Poincar\'e dual of the codimension two invariant submanifold $M_j = \pi^{-1}(F_j)$ in $M$. Note that $H_T^{\ast}(M)$ is not only a ring but also a $H^{\ast}(BT) = \Z[t_1,\ldots,t_n]$-module via the map $p^\ast$, where $p\colon ET \times_T M \to BT$ is the natural projection, and $p^{\ast}$ takes $t_i$ to $\theta_i:=\lambda_{i1}v_1 + \cdots + \lambda_{id}v_d \in \Z(P)$, where $\lambda(F_i)= (\lambda_{1i} , \ldots, \lambda_{ni})^T \in \Z^n$ for $i=1,\ldots, n$.
    Since everything has vanishing odd degrees,
    $H_T^{\ast}(M)$ is a free $H^{\ast}(BT)$-module. Hence the kernel of $\Z(P) = H_T^{\ast}(M) \rightarrow H^{\ast}(M)$ is the ideal $J_\lambda$ of $\Z(P)$ generated by $\theta_1, \ldots, \theta_n$. Therefore, we have
    \begin{equation}\label{eqn:cohomology ring of quasitoric manifold}
        H^{\ast}(M) = \Z[v_1, \ldots, v_d]/(I_P +J_\lambda).
    \end{equation}
    See \cite{DJ} for more details of the previous argument.

    Let $P$ be an $n$-dimensional simple polytope with $d$ facets. Davis and Januszkiewicz \cite{DJ} constructed a $T^d$-manifold $\mathcal{Z}_P$ that is now called the
    \emph{moment angle manifold} of $P$. Let $
    \cF(P)
    =\{F_1, \ldots, F_d\}$ be the set of facets of $P$. For each facet $F_i$ let $T_{F_i}$ denote the one-dimensional coordinate subgroup of $T^{\cF(P)} \cong T^d$ corresponding to $F_i$.
    We assign to every face $F = F_{i_1} \cap \cdots \cap F_{i_\ell}$ the coordinate subtorus
    $$
        T_F = \prod_{j=1}^\ell T_{F_{i_j}} \subset T^d.
    $$ Then the
    moment angle manifold of $P$
    can be constructed as follows:
    $$
        \mathcal{Z}_P=T^d \times P^n / \sim,
    $$ where $(t_1,p)\sim (t_2,q)$ if and only if $p=q$ and $t_1t_2^{-1} \in T_{F(p)}$. From the definition of $\mathcal{Z}_P$, we can
     see easily
     that $\mathcal{Z}_{P_1 \times P_2}=\mathcal{Z}_{P_1} \times \mathcal{Z}_{P_2}$ for any simple polytopes $P_1$ and $P_2$.

    \begin{example}
        It is not so hard to see that the moment angle manifold $\mathcal{Z}_{\Delta^n}$ of an $n$-simplex is homeomorphic to a sphere $S^{2n+1}$, and, hence, $\mathcal{Z}_{\Delta^n \times \Delta^m}=S^{2n+1} \times S^{2m+1}$.
    \end{example}

    Let us fix a characteristic function $\lambda$ on $P$, and let $M(\lambda)$ be the quasitoric manifold as constructed in \eqref{construction:quasitoric}.
    Note that there is a natural identification $\psi_k\colon \Z^k \to \Hom(S^1,T^k)$ given by $(a_1, \ldots, a_k) \mapsto (t \mapsto (t^{a_1},\ldots,t^{a_k}))$ for any positive integer $k$.\linebreak Hence the characteristic matrix $\Lambda$ corresponding to $\lambda$ induces a surjective homomorphism $\overline{\lambda}\colon T^d \rightarrow T^n$\label{lambda} by $\overline{\lambda}(\psi_d(\mathbf{e}_i)(t))=\psi_n(\lambda(F_i))(t)$ for $t \in S^1$, where $\mathbf{e}_i$ is the standard $i$-th basis vector of $\Z^d$ for $i=1,\ldots,d$.
    Then $\ker(\overline{\lambda})$ is a $(d-n)$-dimensional subtorus of $T^{d}$. From the non-singularity condition \eqref{non-singularity}, $\ker(\overline{\lambda})$ meets every isotropy subgroup at the unit. Thus $\ker(\overline{\lambda})$ acts freely on $\mathcal{Z}_P$, and the map $$(\overline{\lambda},id) \colon T^d \times P^n \to T^n \times P^n$$ induces a principal $T^{d-n}$-bundle $\mathcal{Z}_P$ over $M(\lambda)$. We thus have the following proposition.

    \begin{proposition}\cite[Proposition 6.5]{BP}
        The subtorus $\ker(\overline{\lambda})$ acts freely on $\mathcal{Z}_P$, thereby defining a principal $T^{d-n}$-bundle $\mathcal{Z}_P \rightarrow M(\lambda)$.
    \end{proposition}

    Let $M(\lambda_1)$ and $M(\lambda_2)$ be two quasitoric manifolds over a simple polytope $P$. If a self map $\varphi$ of the moment angle manifold $\mathcal{Z}_P$ is $\theta$-equivariant, i.e. there exists an isomorphism $\theta\colon \ker(\overline{\lambda_1}) \rightarrow \ker(\overline{\lambda_2})$ such that $\varphi(t\cdot x)=\theta(t)\cdot \varphi(x)$ for all $t \in \ker(\overline{\lambda_1})$ and $x \in \mathcal{Z}_P$, then there is a natural induced map $\overline{\varphi}$ from $M(\lambda_1)$ to $M(\lambda_2)$:
    \begin{displaymath}
    \xymatrix{
    \mathcal{Z}_P \ar[r]^{\varphi} \ar[d]_{/\ker(\overline{\lambda_1})} & \mathcal{Z}_P \ar[d]^{/\ker(\overline{\lambda_2})} \\
    M(\lambda_1)  \ar[r]_{\overline{\varphi}} & M(\lambda_2)}
    \end{displaymath}
    Thus if we construct a $\theta$-equivariant homeomorphism $\varphi$ from the moment angle manifold $\mathcal{Z}_P$ to itself,  then the induced map $\overline{\varphi}$ is a homeomorphism from $M(\lambda_1)$ to $M(\lambda_2)$.

\section{Quasitoric manifolds with $\betti_2=2$}\label{sec:quasitoric manifolds of $b_2=2$}
    The main interest of the present paper is focused on quasitoric manifolds with the second Betti number $\betti_2=2$. Let $P$ be an $\ell$-dimensional simple polytope with $d$ facets, and let $M$ be a quasitoric manifold over $P$ with the characteristic function $\lambda$. Since $J_\lambda$ consists of $\ell$ linear combinations of $v_1, \ldots, v_d$ and $I_P$ does not contain a linear combination in \eqref{eqn:cohomology ring of quasitoric manifold}, we can see that the second Betti number of $M$ is $d-\ell$. Thus if $P$ supports a quasitoric manifold with $\betti_2=2$, then it has exactly $\ell+2$ facets, and hence $P$ is combinatorially equivalent to a product of two simplices as is well-known, see
    chapter 6 in~\cite{G}. Therefore we may assume that $P = \Delta^n \times \Delta^m$.

    Now consider a quasitoric manifold $M$ of dimension $2(n+m)$
    over $\Delta^n \times \Delta^m$. Consider the facets of $\Delta^n \times \Delta^m$ in the following order: $F_1 \times \Delta^m$, $\ldots$, $F_n \times \Delta^m$, $\Delta^n \times G_1$, $\ldots$, $\Delta^n \times G_m$, $F_{n+1} \times \Delta^m$, $\Delta^n \times G_{m+1}$, where $F_i$'s are the facets of $\Delta^n$ and $G_j$'s are the facets of $\Delta^m$. Then the first $(n+m)$ facets meet at a vertex. Thus, by Remark~\ref{rmk:sign}, the characteristic matrix $\Lambda$ corresponding to $M$ is of the form
    \begin{equation}\label{characteristic matrix}
        \Lambda = (E_{n+m} | \Lambda_\ast) = \left(\begin{array}{llllllcc}
                        1&&&&&&-1&-b_1\\
                        &\ddots&&&0&&\vdots &\vdots\\
                        &&1&&&&-1&-b_n\\
                        &&&1&&&-a_{1}&-1\\
                        &0&&&\ddots&&\vdots &\vdots\\
                        &&&&&1&-a_{m}&-1\\
                        \end{array}\right)
    \end{equation}
    up to equivalence, where $1-a_jb_i=\pm 1$ for $i=1,\ldots,n$ and $j=1,\ldots,m$ because of the non-singularity condition \eqref{non-singularity} of the characteristic function. From now on we denote such $M$ by $M_{\va,\bb}$\label{M} for $\va=(a_1,\ldots,a_m)$ and $\bb=(b_1,\ldots,b_n)$.
    Hence, from~\eqref{eqn:cohomology ring of quasitoric manifold},
    the cohomology ring of $M_{\va,\bb}$ with $\Z$ coefficients is
    \begin{equation}\label{cohomology ring}
     H^{\ast}(M_{\va,\bb})
        =\Z[x_1,x_2]/ \langle x_1\prod_{i=1}^{n}(x_1+b_ix_2),\,x_2\prod_{j=1}^{m}(a_{j}x_1+x_2) \rangle.
    \end{equation}

    A (complex) \emph{generalized Bott tower} of height $h$, or an \emph{$h$-stage generalized Bott tower}, is a sequence
$$
        B_h\stackrel{\pi_h}\longrightarrow B_{h-1} \stackrel{\pi_{h-1}}\longrightarrow
        \dots \stackrel{\pi_2}\longrightarrow B_1 \stackrel{\pi_1}\longrightarrow
        B_0=\{\text{a point}\}
$$
    of manifolds $B_i = P(\underline{\C}\oplus \bigoplus_{j=1}^{\ell_i} \xi_{i, j})$, where $\xi_{i,j}$ is a complex line bundle, $\underline{\C}$ is the trivial complex line bundle over $B_{i-1}$ for each $i=1, \ldots, h$, and $P(\cdot)$ stands for the projectivization. We call \emph{$B_i$ the $i$-stage generalized Bott manifold}.

    Note that the Whitney sum of $\ell$ complex line bundles admits
    a canonical $T^\ell$-action. Assume $B_{j-1}$ admits an effective $T^{\sum_{k=1}^{j-1} \ell_k}$-action. Since $H^1(B_{j-1})=0$, it lifts to an action on $\xi_i$, see \cite{HY}. Moreover, it commutes with the canonical $T^{\ell_i}$-action on $\xi_i$, and hence, it induces an effective $T^{\sum_{k=1}^{j} \ell_k}$-action on $B_j$. Thus, we can define an effective half-dimensional torus action on $B_h$ inductively. One can show that this action is locally standard and its orbit space is a product of $h$ simplices $\prod_{i=1}^h \Delta^{\ell_i}$. Thus a two-stage generalized Bott manifold is a quasitoric manifold over $P=\Delta^{\ell_1}\times \Delta^{\ell_2}$ and has $\betti_2=2$.

    \begin{remark}
    In fact, a generalized Bott manifold is not only a quasitoric manifold but also a (projective) toric manifold. Note that all toric manifolds admit $T^n$-invariant complex structures. Hence, by \cite[Theorem 6.4]{ch-ma-su08}, all toric manifolds over a product of simplices are generalized Bott manifolds.
    \end{remark}

    We already know a necessary and sufficient condition for a quasitoric manifold $M$ to be equivalent to a generalized Bott manifold by the following proposition.

    \begin{proposition}\cite{ch-ma-su08} \label{prop:condition to be GBM}
    Let $M$ be a quasitoric manifold over $P=\prod_{i=1}^h\Delta^{\ell_i}$, and let $\Lambda_\ast$ be an $h\times h$ vector matrix associated with $M$.\footnote{In fact, $\Lambda_\ast$ is a $(\sum_{i=1}^h \ell_i) \times h$ matrix. Then $\Lambda_\ast$ can be viewed as an $h\times h$ vector matrix whose entries in the $i$-th row are vectors in $\Z^{\ell_i}$. A more precise description of (a transpose version of) $\Lambda_\ast$ is explained on page 114 in~\cite{ch-ma-su08}.} Then $M$ is equivalent to a generalized Bott manifold if and only if $\Lambda_\ast$ is conjugate to an $h\times h$ lower triangular vector matrix.
    \end{proposition}

    Moreover, the following theorem gives a smooth classification of two-stage generalized Bott manifolds.

    \begin{theorem}\cite{CMS}\label{thm:classification of CMS}
        Let $B_2=P(\oplus_{i=0}^m\gamma^{u_i})$ and $B_2^\prime=P(\oplus_{i=0}^m\gamma^{u_i^\prime})$, where $u_0=u_0'=0$ and $\gamma^{u_i}$ denotes the complex line bundle over $B_1=\CP^n$ whose first Chern class is $u_i \in H^2(B_1)$. Then the following are equivalent.
        \begin{enumerate}
            \item There exists $\epsilon=\pm1$ and $w \in H^2(B_1)$ such that
                $$\prod_{i=0}^m(1+\epsilon(u_i^\prime+w))=\prod_{i=0}^m(1+u_i) \mbox{ in } H^\ast(B_1).$$
            \item $B_2$ and $B_2^\prime$ are diffeomorphic.
            \item $H^\ast(B_2)$ and $H^\ast(B_2^\prime)$ are isomorphic as graded rings.
        \end{enumerate}
    \end{theorem}

    When a quasitoric manifold  $M$ is equivalent to a two-stage generalized Bott manifold,
    we may assume that  $M=M_{\va,\mathbf{0}}$.
    In this case, $M$ is a $\CP^m$-bundle over $\CP^n$, and $H^{\ast}(M_{\va,\mathbf{0}})$ is of the form
    \begin{equation}\label{cohomology ring of GB}
        H^{\ast}(M_{\va,\mathbf{0}})
        =\Z[x_1,x_2]/ \langle {x_1}^{n+1},\,x_2\prod_{j=1}^{m}(a_{j}x_1+x_2) \rangle.
    \end{equation}

    If a quasitoric manifold $M$ with $\betti_2=2$ is not equivalent to a generalized Bott manifold, then we may assume that $M=M_{\va,\bb}$ for some nonzero vectors $\va$ and $\bb$
    by Proposition~\ref{prop:condition to be GBM}.
    Then $a_jb_i=2$ for some $i$ and $j$.
     Without loss of generality, we may assume that $a_{j}$ is $0$ or $\pm2$ and $b_i$ is $0$ or $\pm 1$.
    Note that the signs of nonzero $a_j$'s and $b_i$'s are the same and, by Remark~\ref{rmk:sign}, $M_{\va,\bb}$ is equivalent to $M_{-\va,-\bb}$.\footnote{We can see easily by the following steps; 1) change the signs of the first $n$ row vectors of the characteristic matrix~\eqref{characteristic matrix}, 2) change the signs of the first $n$ column vectors and the $(n+m+2)$-nd of the resulting matrix. Then we can obtain the characteristic matrix corresponding to $M_{-\va,-\bb}$.} Hence, we may assume that the nonzero $a_j$ is $2$, and the nonzero $b_i$ is $1$. Now let $s$ be the number of $a_{j}=2$ for $j=1,\ldots,m$ and $r$ the number of $b_i=1$ for $i=1,\ldots,n$.
    Then, the cohomology ring of $M_{\va,\bb}$ is isomorphic to
    \begin{equation}\label{cohomology ring of NGB}
        \Z[x_1,x_2]/\langle x_1^{n+1-r}(x_1 + x_2)^{r},\,x_2^{m+1-s}(2x_1+x_2)^{s}\rangle
    \end{equation}
    for some $1 \leq r \leq n$ and $1 \leq s \leq m$.

    We close this section by giving another construction of quasitoric manifolds over $\Delta^n \times \Delta^m$ from the moment angle manifold $\mathcal{Z}_{\Delta^n \times \Delta^m}$.
    \begin{remark}\label{free action}
        Note that the moment angle manifold $\mathcal{Z}_{\Delta^n \times \Delta^m}$ is
        $$S^{2n+1}\times S^{2m+1}=\{(\vw,\vz) \in \C^{n+1} \times \C^{m+1} \colon |\vw|=1,~|\vz|=1\},$$
        which has the standard $T^{n+m+2}$-action of the componentwise complex multiplication. Let $\lambda$ be a characteristic function corresponding to $M_{\va,\bb}$,
        and let $K_{\va,\bb}$ be the image of the homomorphism $\mu: T^2 \to T^{n+m+2}$ defined by
        \begin{equation}\label{eqn:homomorphism}
            \left(\begin{array}{cc}1&b_1\\\vdots&\vdots\\1&b_n\\1&0\\a_1&1\\\vdots&\vdots\\a_m&1\\0&1\end{array}\right).
        \end{equation}
        Then the action of the two-torus $K_{\va,\bb}$ on $S^{2n+1}\times S^{2m+1}$ defined by
        \begin{equation*}
        \begin{split}
            &\mu(t_1,t_2)\cdot ((w_1,\ldots,w_{n+1}),(z_1,\ldots,z_{m+1}))\\
            &=((t_1t_2^{b_1}w_1,\ldots,t_1t_2^{b_n}w_n,t_1w_{n+1}),(t_1^{a_{1}}t_2z_1,\ldots,t_1^{a_{m}}t_2z_m,t_2z_{m+1}))
        \end{split}
        \end{equation*}
        is free because of the non-singularity condition \eqref{non-singularity} of $\lambda$. Moreover,
        this action is exactly equal to the $\ker(\overline{\lambda})$-action on $\mathcal{Z}_{\Delta^n \times \Delta^m}$, where a homomorphism $\overline{\lambda}$ is defined
        on page~\pageref{lambda}, and
        the quasitoric manifold $M_{\va,\bb}$ is the orbit space $\mathcal{Z}_{\Delta^n \times \Delta^m}/K_{\va,\bb}$ with the action of $T^{n+m}$ defined by
        \begin{equation*}
        \begin{split}
            &(t_1,\ldots,t_{n+m})\cdot [(w_1,\ldots,w_{n+1}),(z_1,\ldots,z_{m+1})]\\
            &=[(t_1w_1,\ldots,t_nw_n,w_{n+1}),(t_{n+1}z_1,\ldots,t_{n+m}z_m,z_{m+1})].
        \end{split}
        \end{equation*}
        See \cite{ch-ma-su08} for more details.

        In other words, the subtorus $K_{\va,\bb} \subset T^{n+m+2}$ is represented by the unimodular subgroup of $\Z^{n+m+2}$ spanned by the two vectors $(1,\ldots,1,a_{1},\ldots,a_{m},0)$ and $(b_1,\ldots,b_n,0,1,\ldots,1)$. Note that these two vectors generate the null space of the matrix
        \begin{equation}\label{char matrix by reordering}
            \left(\begin{array}{lllclllc}
                        1&&&-1&&&&-b_1\\
                        &\ddots&&\vdots&0&& &\vdots\\
                        &&1&-1&&&&-b_n\\
                        &&&-a_{1}&1&&&-1\\
                        &0&&\vdots&&\ddots& &\vdots\\
                        &&&-a_{m}&&&1&-1\\
                        \end{array}\right)
        \end{equation}
        which is obtained from $\Lambda$ in \eqref{characteristic matrix} by changing the ordering of facets of $\Delta^n \times \Delta^m$: $F_1 \times \Delta^m$, $\ldots$, $F_n \times \Delta^m$, $F_{n+1} \times \Delta^m$, $\Delta^n \times G_1$, $\ldots$, $\Delta^n \times G_m$, $\Delta^n \times G_{m+1}$.
    \end{remark}

\section{Quasitoric manifolds equivalent to a generalized Bott manifold}\label{sec:quasitoric mfds equiv to GBM}

    Recall that the cohomological rigidity problem
    asks whether two quasitoric manifolds are homeomorphic if their cohomology rings are isomorphic. As an intermediate step toward the
    answer to
    the question for quasitoric manifolds homeomorphic to generalized Bott manifolds, we can ask the following question: \emph{is a quasitoric manifold over a product of simplices equivalent (or homeomorphic) to a generalized Bott manifold if its cohomology ring is isomorphic to that of a generalized Bott manifold?} When the orbit space is $\Delta^1\times \Delta^1$, then the answer is
    affirmative by \cite{CS}.
    Assume that a two-stage generalized Bott manifold is a $\CP^m$-bundle over $\CP^n$. In this section we answer to this question for $m>1$ case.
    For the case of $m=1$, we will show in the next section that if a quasitoric manifold $M$ has the cohomology ring of the type~\eqref{cohomology ring of GB}, then $M$ is homeomorphic (but not necessarily equivalent) to a generalized Bott manifold.

    \begin{proposition}\label{prop:GB_int}
        Let $M$ be a quasitoric manifold over  $\Delta^n \times \Delta^m$ with $m>1$. If there is a generalized Bott tower $B_2 \to \CP^n \to B_0$ such that the fiber of $B_2 \to \CP^n$ has complex dimension $m$ and $H^\ast(B_2) \cong H^\ast(M)$, then $M$ is equivalent to a generalized Bott manifold.
    \end{proposition}

    \begin{proof}
        From~\eqref{cohomology ring of GB}, the cohomology ring of $B_2$ can be given by
        \begin{equation*}
        H^\ast(B_2)=\Z[x_1,x_2]/ \langle {x_1}^{n+1},\,x_2\prod_{j=1}^{m}(a_{j}x_1+x_2) \rangle,
        \end{equation*}
        and from~\eqref{cohomology ring}, the cohomology ring of $M$ can be given by
        \begin{equation*}
        H^{\ast}(M)
        =\Z[y_1,y_2]/\langle y_1\prod_{i=1}^{n}(y_1+d_i\,y_2),\,y_2\prod_{j=1}^{m}(c_{j}\,y_1+y_2)\rangle.
        \end{equation*}
        For simplicity, let $\mathcal{I} \subset \Z[x_1,x_2]$ be the ideal generated by the homogeneous polynomials ${x_1}^{n+1}$ and $x_2\prod_{j=1}^{m}(a_{j}x_1+x_2)$ and let $\mathcal{J} \subset \Z[y_1,y_2]$ be also the ideal generated by the homogeneous polynomials $y_1\prod_{i=1}^{n}(y_1+d_i\,y_2)$ and $y_2\prod_{j=1}^{m}(c_{j}\,y_1+y_2)$. Then we have
        $H^\ast(B_2)=\Z[x_1,x_2]/\mathcal{I}$ and $H^\ast(M)=\Z[y_1,y_2]/\mathcal{J}.$

        From the hypothesis, there is a ring isomorphism $\phi: H^\ast(B_2) \to H^\ast(M)$ which preserves the grading. Then the map $\phi$ lifts to a grading preserving isomorphism
        $\bar{\phi}: \Z[x_1,x_2] \to \Z[y_1,y_2]$
        with $\bar{\phi}(\mathcal{I})=\mathcal{J}$. Note that if we let $\bar{\phi}(x_i)=g_{i1}y_1+g_{i2}y_2$, $i=1,2$, then the determinant of $G=\left(
               \begin{array}{cc}
                 g_{11} & g_{12} \\
                 g_{21} & g_{22} \\
               \end{array}
             \right)$ is $\pm 1$.

        We prove the proposition by showing that either $c_{1}=\cdots=c_{m}=0$ or $d_1=\cdots=d_n=0$. Then, by Proposition~\ref{prop:condition to be GBM}, $M$ is equivalent to a generalized Bott manifolds. We consider three cases (1) $n < m$, (2) $n=m$, and (3) $1 < m < n$ separately.

        \textbf{CASE 1: $n < m$}

        Since $\bar{\phi}(\mathcal{I})=\mathcal{J}$ and $m>n$, we have
        \begin{equation}\label{subcase 1-1 added}
            \bar{\phi}(x_1^{n+1})=\alpha y_1\prod_{i=1}^{n}(y_1+d_i\,y_2),
        \end{equation}
        where $\alpha$ is an integer. Then
        the set of prime divisors of $x_1^{n+1}$ is mapped by $\bar{\phi}$ to the set of prime divisors of $\alpha y_1\prod_{i=1}^{n}(y_1+d_i\,y_2)$. Since $x_1$ is the only prime divisor of $x_1^{n+1}$, we must have $\alpha \neq 0$ and $d_i=0$ for all $i$, which prove the proposition in this case.

        \textbf{CASE 2: $n = m$}

        Since $\bar{\phi}(\mathcal{I})=\mathcal{J}$ and $n=m$, we have
        \begin{equation}\label{identity when n=m added added}
            \bar{\phi}(x_1^{n+1})=\alpha y_1\prod_{i=1}^{n}(y_1+d_i\,y_2)+ \alpha^\prime y_2\prod_{j=1}^{n}(c_{j}\,y_1+y_2),
        \end{equation}
        where $\alpha$ and $\alpha^\prime$ are integers.

        \textbf{(i)} If $\alpha$ is zero, then from the similar argument to the case 1, we have $\alpha^\prime \neq 0$ and $c_j=0$ for all $j$, which prove the proposition.

        \textbf{(ii)} If $\alpha^\prime$ is zero, then the similar argument shows that $\alpha \neq 0$ and $d_i=0$ for all $i$, which prove the proposition.

        \textbf{(iii)} Now assume that neither $\alpha$ nor $\alpha^\prime$ is zero. Plugging $\bar{\phi}(x_1)=g_{11}y_1+g_{12}y_2$ into~\eqref{identity when n=m added added}, we have
        \begin{equation}\label{identity when n=m added}
            (g_{11}y_1+g_{12}y_2)^{n+1}=\alpha y_1\prod_{i=1}^{n}(y_1+d_i\,y_2)+ \alpha^\prime y_2\prod_{j=1}^{n}(c_{j}\,y_1+y_2).
        \end{equation} Then
        we can see that $\alpha=g_{11}^{n+1}$ and $\alpha^\prime=g_{12}^{n+1}$ by comparing the coefficients of $y_1^{n+1}$ and $y_2^{n+1}$ on both sides of~\eqref{identity when n=m added}. Hence we have
        \begin{equation}\label{identity when n=m}
            (g_{11}y_1+g_{12}y_2)^{n+1}=g_{11}^{n+1}y_1\prod_{i=1}^{n}(y_1+d_i\,y_2)+g_{12}^{n+1}y_2\prod_{j=1}^{n}(c_{j}\,y_1+y_2)
        \end{equation}
        as polynomials. Plug $y_1=y_2=1$ and $y_1=1,y_2=-1$ into~\eqref{identity when n=m} to get the following system of equations
        \begin{equation}\label{1st system of eqns when n=m}
        \begin{split}
            (g_{11}+g_{12})^{n+1}&=g_{11}^{n+1}\prod_{i=1}^{n}(1+d_i)+g_{12}^{n+1}\prod_{j=1}^{n}(c_{j}+1)\\
            (g_{11}-g_{12})^{n+1}&=g_{11}^{n+1}\prod_{i=1}^{n}(1-d_i)-g_{12}^{n+1}\prod_{j=1}^{n}(c_{j}-1)
        \end{split}
        \end{equation}
        Note that $1-d_ic_{j}=\pm1$ for all $1\leq i,j \leq n$ from the non-singularity condition \eqref{non-singularity}.
        If we show that $d_ic_{j}=0$ for all $1 \leq i,j \leq n$, then we are done. Indeed, if $c_{j_0} \neq 0$ for some $1 \leq j_0 \leq n$, then $d_ic_{j_0}=0$ for all $1 \leq i \leq n$ implies that $d_i=0$ for all $1 \leq i \leq n$. Otherwise $c_{j}=0$ for all $1 \leq j \leq n$, which proves the proposition.

        \textbf{We now show that $d_ic_{j}=0$ for all $1 \leq i,j \leq n$.} Suppose not, i.e., $d_{i_0}c_{j_0} \neq 0$ for some $1 \leq i_0, j_0 \leq n$. Then from the non-singularity condition we have $d_{i_0}c_{j_0}=2$. For simplicity we may assume that $d_1c_{1}=2$,
        \underline{$d_1=1$ and $c_{1}=2$} or \underline{$d_1=2$ and $c_{1}=1$} up to equivalence. But these two cases are symmetric because $n=m$. Thus it is enough to consider the case $d_1=1$ and $c_{1}=2$.

        Since $1-d_ic_{j}=\pm 1$ for all $1 \leq i,j \leq n$, we have that $d_i=1$ or $0$ for $i=2,\ldots,n$ and $c_{j}=2$ or $0$ for $j=2,\ldots,n$. Plug these into~\eqref{1st system of eqns when n=m} to get
        \begin{align}
            (g_{11}+g_{12})^{n+1}&=2^{r}g_{11}^{n+1}+3^{s}g_{12}^{n+1} \label{1 of 2nd system when n=m} \mbox{ and}\\
            (g_{11}-g_{12})^{n+1}&=-g_{12}^{n+1}(-1)^{n+1-s} \label{2 of 2nd system when n=m}
        \end{align}
        for some $1 \leq r,\,s \leq n$. From~\eqref{2 of 2nd system when n=m}, we have $g_{11}=0$ or $g_{11}=2g_{12}$. If $g_{11}=0$, then~\eqref{1 of 2nd system when n=m} implies $g_{12}^{n+1}=3^{s}g_{12}^{n+1}$. Therefore $g_{12}=0$, which contradicts to $\det(G)\neq 0$.
        Otherwise, i.e. $g_{11}=2g_{12}$, then by plugging $g_{11}=2g_{12}$ into~\eqref{1 of 2nd system when n=m} we have $$3^{n+1}g_{12}^{n+1}=2^{r+n+1}g_{12}^{n+1}+3^sg_{12}^{n+1}.$$ Therefore, $g_{12}$ is zero and so is $g_{11}$, which also contradicts to $\det(G) \neq 0$.
        This is because we assumed that $d_{i_0}c_{j_0} =2$ for some $1 \leq i_0,j_0 \leq n$. This shows that $d_{i}c_{j}=0$ for all $1 \leq i,j \leq n$.

        \textbf{CASE 3: $1 < m < n$}

        Since $n>m$, we have
        \begin{equation}\label{1st equation when m<n added added}
        \bar{\phi}(x_2\prod_{j=1}^m(a_jx_1+x_2))=\alpha y_2\prod_{j=1}^m(c_{j}\,y_1+y_2)
        \end{equation}
        for some nonzero integer $\alpha$. Plugging $\bar{\phi}(x_i)=g_{i1}y_1+g_{i2}y_2$ into~\eqref{1st equation when m<n added added}, we have
        \begin{equation}\label{1st equation when m<n added}
            (g_{21}y_1+g_{22}y_2)\prod_{j=1}^m((a_{j}g_{11}+g_{21})y_1+(a_{j}g_{12}+g_{22})y_2)
            =\alpha y_2\prod_{j=1}^m(c_{j}\,y_1+y_2).
        \end{equation}
        Comparing the coefficients of $y_2^{n+1}$ on both sides of~\eqref{1st equation when m<n added}, we can see that $\alpha=g_{22}\prod_{j=1}^m(a_{j}g_{12}+g_{22})$ and we have
        \begin{equation}\label{1st equation when m<n}
            \begin{split}
            (g_{21}y_1+g_{22}y_2)\prod_{j=1}^m((a_{j}g_{11}+g_{21})y_1+(a_{j}g_{12}+g_{22})y_2)&\\
            \qquad=g_{22}\prod_{j=1}^m(a_{j}g_{12}+g_{22})y_2\prod_{j=1}^m(c_{j}\,y_1+y_2).&\\
            \end{split}
        \end{equation}
        By comparing the coefficients of $y_1^{m+1}$ on both sides of \eqref{1st equation when m<n}, we have $g_{21}\prod_{j=1}^m(a_{j}g_{11}+g_{21}) = 0$. If $g_{21}=0$, then $\det(G)=g_{11}g_{22}=\pm1$, and hence $g_{11}=\pm1$. If $a_{j}g_{11}+g_{21}=0$ for some $1 \leq j \leq m$, then $\det(G)=g_{11}g_{22}-g_{12}g_{21}=g_{11}(g_{22}+a_{j}g_{12})=\pm1$. Hence $g_{11}=\pm1$, too.

        As in case 2, it is enough to show that $d_ic_{j}=0$ for all $1 \leq i \leq n$ and $1 \leq j \leq m$.
        Suppose otherwise, i.e. $d_{i_0}c_{j_0}=2$ as before.

        \textbf{(i)} Suppose $c_{j_0}=2$. Then $d_{i_0}=1$, and $c_{j}=0$ or $2$ for all $1 \leq j \leq m$ and $d_i=0$ or $1$ for all $1 \leq i \leq n$. Let $s$ be the number of $c_{j}$'s equal to $2$.

        \textbf{(i-1)} First consider the case $0<s<m$. In this case we may assume $c_{1}=2$ and $c_{m}=0$ for simplicity.
        Since $\bar{\phi}(x_1^{n+1})\in\mathcal{J}$, we have
        \begin{equation}\label{system when n>m added added}
        \bar{\phi}(x_1^{n+1})=\alpha y_1\prod_{i=1}^n(y_1+d_iy_2)+f(y_1,y_2)y_2\prod_{j=1}^m(c_{j}y_1+y_2),
        \end{equation}
        where $\alpha$ is an integer and $f(y_1,y_2)$ is a homogeneous polynomial of degree $n-m$. Plugging $\bar{\phi}(x_1)=g_{11}y_1+g_{12}y_2$ into~\eqref{system when n>m added added}, we have
        \begin{equation}\label{system when n>m added}
            \begin{split}
            &(g_{11}y_1+g_{12}y_2)^{n+1}\\
            &\qquad=\alpha y_1\prod_{i=1}^n(y_1+d_iy_2)+f(y_1,y_2)y_2\prod_{j=1}^m(c_{j}y_1+y_2).\\
            \end{split}
        \end{equation}
        If $\alpha=0$, then $g_{11}=0$, so $c_j=0$ for all $j=1,\ldots,m$. This is a contradiction to the assumption $c_1=1$. Hence $\alpha \neq 0$.
        Comparing the coefficients of $y_1^{n+1}$ on both sides of~\eqref{system when n>m added}, we can see that $\alpha=g_{11}^{n+1}$ and we have
        \begin{equation}\label{system when n>m}
            \begin{split}
            &(g_{11}y_1+g_{12}y_2)^{n+1}\\
            &\qquad=g_{11}^{n+1}y_1\prod_{i=1}^n(y_1+d_iy_2)+f(y_1,y_2)y_2\prod_{j=1}^m(c_{j}y_1+y_2),\\
            \end{split}
        \end{equation}
        as polynomials in $y_1$ and $y_2$. Since $c_{m}=0$, comparing the coefficients of $y_1^{n}y_2$ on both sides of~\eqref{system when n>m}, we get the equation
        \begin{equation}\label{eqn in case 2}
            (n+1)g_{11}^ng_{12}=g_{11}^{n+1}(d_1+\cdots+d_n).
        \end{equation}
        Since $g_{11}=\pm 1$ and $d_i=0$ or $1$ with $d_1+\cdots+d_n \leq n$, the last equation gives a contradiction. So $s<m$ cannot happen.

        \textbf{(i-2)} Now suppose $s=m$, i.e., $c_{1}=\cdots=c_{m}=2$. In this case there is a ring isomorphism $\psi$ from the cohomology ring $$H^{\ast}(M)=\Z[y_1,y_2]/\langle y_1^{n+1-r}(y_1+y_2)^{r},y_2(2y_1+y_2)^m \rangle$$ to the ring $\Z[Y_1,Y_2]/\langle Y_1^{n+1-r}(Y_1+Y_2)^{r},Y_2^{m}(2Y_1+Y_2)\rangle$ given by $\psi(y_1)=-Y_1$, $\psi(y_2)=2Y_1+Y_2$.
        In other words, if $s=m$, then $H^{\ast}(M)$ is isomorphic to a ring
        $$\Z[y_1,y_2]/\langle y_1 \prod_{i=1}^n(y_1+d_iy_2), y_2\prod_{j=1}^m(c_jy_1+y_2)\rangle$$
        with $c_{1}=2$, $c_{2}=\cdots=c_{m}=0$, i.e., $s=1$ case. But by the previous argument this induces a contradiction.

        \textbf{(ii)} Suppose $c_{j_0}=1$. Then $d_{i_0}=2$. As before let $r$ be the number of $c_{j}$'s equal to 1.

        \textbf{(ii-1)} First consider the case when $0 < r <m$. In this case we may assume that $c_{1}=1$ and $c_{m}=0$. By the same argument as above, \eqref{system when n>m} and~\eqref{eqn in case 2} also hold.
         Since $g_{11}=\pm 1$, we have
        $
            (n+1)g_{12}=g_{11}(d_1+\cdots + d_n)
            =2g_{11}s,
        $
        where $s$ is the number of $d_i$'s equal to $2$, and $0 < s \leq n$. This equality holds if and only if $g_{11}=g_{12}$, $s=\frac{n+1}{2}$, and $n$ is odd. By plugging $y_1=1$ and $y_2=-1$ into~\eqref{system when n>m}, we have $0=g_{11}^{n+1}\prod_{i=1}^n(1-d_i)$ which is a contradiction. This shows that $0 < r <m$ is impossible.

        \textbf{(ii-2)} Now suppose $r=m$, i.e., $c_{1}=\cdots=c_{m}=1$. Then by the ring isomorphism given by $\psi(y_1)=-Y_1$ and $\psi(y_2)=Y_1+Y_2$, $H^{\ast}(M)$ is isomorphic to the ring
        $$\Z[Y_1,Y_2]/\langle Y_1^{n+1-s}(Y_1+2Y_2)^{s},Y_2^m(Y_1+Y_2)\rangle,$$
        which is the case when $r=1$. By the previous argument, this case also induces a contradiction.

        We thus have proved that $d_ic_{j}=0$ for all $1 \leq i \leq n$, $1 \leq j \leq m$, which proves the proposition.

    \end{proof}

\section{Quasitoric manifolds over $\Delta^n \times \Delta^1$}\label{sec:classification of quasitoric manifolds when m=1}

    In this section, we restrict our attention to the case where the orbit space is $\Delta^n \times \Delta^1$.

    \begin{example}\cite{DJ}\label{exam:hirz}
        Projective toric manifolds over $\Delta^1 \times \Delta^1$ are \emph{Hirzebruch surfaces} $\Sigma_a = P(\underline{\C}\oplus\gamma^{\otimes a})$ for $a \in \Z$, where $\gamma$ is the tautological line bundle over $\CP^1$. By \cite{Hirzebruch-1951}, $\Sigma_a$ is diffeomorphic to $\Sigma_b$ if and only if $a$ is congruent to $b$ modulo $2$. Hence a projective toric manifold over $\Delta^1 \times \Delta^1$ is diffeomorphic to $\CP^1 \times \CP^1$ or $\CP^2\# \overline{\CP^2}$. On the other hand, $\CP^2 \# \CP^2$ is the unique quasitoric manifold over $\Delta^1 \times \Delta^1$ which is not a projective toric manifold. Hence there are only three
        topological
        types of quasitoric manifolds over $\Delta^1 \times \Delta^1$: $\CP^1 \times \CP^1$, $\CP^2 \# \overline{\CP^2}$, and $\CP^2 \# \CP^2$.
    \end{example}

    Let $M$ be a quasitoric manifold over $\Delta^n \times \Delta^1$. As in Section~\ref{sec:quasitoric manifolds of $b_2=2$}, we order the facets of $\Delta^n \times \Delta^1$ as follows:
        \begin{equation}\label{ordering of facets}
            F_1 \times \Delta^1, \ldots, F_n \times \Delta^1, \Delta^n \times G_1, F_{n+1} \times \Delta^1, \Delta^n \times G_2,
        \end{equation}
    where $F_i$'s are facets of $\Delta^n$ and $G_i$ are facets of $\Delta^1$. Up to equivalence of quasitoric manifolds we may assume that the characteristic function $\lambda$ on the ordered facets gives the following $(n+1) \times (n+3)$ matrix
    \begin{equation}\label{characteristic matrix for m=1}
        \Lambda=\left(\begin{array}{cccccc}
        1&\cdots&0&0&-1&-b_1\\
        \vdots&\ddots&\vdots&\vdots&\vdots&\vdots\\
        0&\cdots&1&0&-1&-b_n\\
        0&\cdots&0&1&-a&-1\\
        \end{array}\right),
    \end{equation}
    namely, $\lambda(F_i \times \Delta^1)
    =\mathbf{e}_i$ for $0 \leq i \leq n$, $\lambda(\Delta^n \times G_1)
    =\mathbf{e}_{n+1}$, $\lambda(F_{n+1} \times \Delta^1)=(-1,\ldots,-1,-a)^T$, and $\lambda(\Delta^n \times G_2)=(-b_1,\ldots,-b_n,-1)^T$.
    We denote such $M$ by $M_{a,\mathbf{b}}$ for $\mathbf{b}=(b_1,\ldots,b_n) \in \Z^n$. Moreover, by the non-singularity condition \eqref{non-singularity}, we have $ab_i=0$ or $2$ for $i=1, \ldots, n$.

    We first consider the case $ab_i=0$ for all $i=1, \ldots, n$. Then $a=0$ or $(b_1,\ldots,b_n)$ is a zero vector. Then $M_{a,\bb}$ is equivalent to a generalized Bott manifold by Proposition~\ref{prop:condition to be GBM}. More precisely, $M_{a,\mathbf{0}} = P(\underline{\C} \oplus \gamma^{\otimes a}) \to \CP^n$, and $M_{0,\bb} = P(\underline{\C}\oplus(\bigoplus_{j=1}^n \gamma^{\otimes b_j})) \to \CP^1$. In this case, $M_{a,\bb}$ is a projective toric manifold. Here, we classify all projective toric manifolds over $\Delta^n \times \Delta^1$ smoothly.

    \begin{proposition}\label{classification of GB with m=1}
        Let $n$ be a positive integer greater than $1$.
        \begin{enumerate}
            \item Let $M_{a, \mathbf{0}}$ denote the two-stage generalized Bott manifold
                $$
                M_{a, \mathbf{0}} = B_2 \stackrel{\pi_2}\longrightarrow B_1 \stackrel{\pi_1}\longrightarrow
                B_0=\{\text{a point}\},
                $$
                where $B_1=\CP^n$, $B_2=P(\underline{\C}\oplus \gamma^{\otimes a})$, and $\gamma$ is the tautological line bundle over $\CP^n$. Then $M_{a, \mathbf{0}}$ is diffeomorphic to $M_{a', \mathbf{0}}$ if and only if $|a|=|a^{\prime}|$.
            \item Let $M_{0,\mathbf{b}}$ denote the two-stage generalized Bott manifold
                $$
                M_{0,\mathbf{b}} = B_2 \stackrel{\pi_2}\longrightarrow B_1 \stackrel{\pi_1}\longrightarrow
                B_0=\{\text{a point}\},
                $$
                where $B_1=\CP^1$, $B_2=P(\underline{\C}\oplus (\bigoplus_{i=1}^n\gamma^{\otimes b_i}))$ for $\mathbf{b}=(b_1,\ldots,b_n)\in\Z^n$, and $\gamma$ is the tautological line bundle over $\CP^1$.
                Then $M_{0,\mathbf{b}}$ is diffeomorphic to $M_{0,\mathbf{b}^{\prime}}$ if and only if there is $\epsilon= \pm1$ such that $\epsilon\sum_{i=1}^n b_i \equiv \sum_{i=1}^n b_i^{\prime} ~(\mathrm{mod}~n+1)$.
        \end{enumerate}
    \end{proposition}
    \begin{proof}
    (1)
        Note that $H^\ast(M_{a, \mathbf{0}})=\Z[x_1,x_2]/\langle x_1^{n+1}, x_2(ax_1+x_2) \rangle$, and
        $\pi_2^\ast(H^\ast(B_1)) \cong \Z[x_1]/x_1^{n+1} \subset
        H^\ast(M_{a,\mathbf{0}})$.
        By Theorem~\ref{thm:classification of CMS}, $M_{a, \mathbf{0}}$ and $M_{a', \mathbf{0}}$ are diffeomorphic if and only if there exist $\epsilon = \pm 1$ and $w \in \Z$ such that
        $$
        (1 + \epsilon wx_1) (1+ \epsilon (a + w)x_1) = (1 + a'x_1)
        $$ in $\Z[x_1]/x_1^{n+1}$. Hence, we have $\epsilon(a + 2w) = a'$ and $w(a+w) x_1^2 = 0$. Since $n>1$, $x_1^2 \neq 0$ in $\Z[x_1]/x_1^{n+1}$. Therefore $w(a+w)=0$, hence $w$ is either $0$ or $-a$. In any case, we obtain $a' = \pm a$.

    (2)
        Note that $H^\ast(M_{0,\mathbf{b}}) = \Z[x_1,x_2]/\langle x_1\prod_{i=1}^n(x_1+b_i x_2), x_2^2 \rangle$ and $\pi_2^\ast(H^\ast(B_1)) \cong \Z[x_2]/x_2^2 \subset
        H^\ast(M_{0,\mathbf{b}})$.
        By Theorem~\ref{thm:classification of CMS}, $M_{0,\mathbf{b}}$ and $M_{0,\mathbf{b'}}$ are diffeomorphic if and only if there exist $\epsilon = \pm 1$ and $w \in \Z$ such that
        $$
        \prod_{i=0}^n (1+ \epsilon(b_i + w)x_2) = \prod_{i=0}^n (1+ b_i'x_2)
        $$  in $\Z[x_2]/x_2^2$, where $b_0 = b'_0 = 0$. Since $x_2^2 =0$ we only have the condition $\epsilon \sum_{i=1}^n b_i + (n+1)w = \sum_{i=1}^n b'_i$.
    \end{proof}

    Now, we consider the case $ab_i=2$ for some $i$. In this case, $M_{a,\mathbf{b}}$ cannot be equivalent to a generalized Bott manifold. However, as we will see later, they can be homeomorphic to generalized Bott manifolds. Note that, by Remark~\ref{rmk:sign}, we may assume that $a$ and the nonzero $b_i$'s have the positive sign. If $ab_i=2$ for some $i=1,\ldots,n$, then $a$ must be either $1$ or $2$.

    Let $s$ be the number of the nonzero $b_i$'s. Then, by \eqref{cohomology ring}, we have
    $$
    H^\ast(M_{a,\mathbf{b}})
    = \Z[x_1, x_2]/ \langle x_1^{n+1-s} (x_1 + bx_2)^s, x_2(ax_1 + x_2)\rangle,
    $$ where $ab=2$.

    Here, we classify all quasitoric manifold which are not equivalent to projective toric manifolds over $\Delta^n \times \Delta^1$ topologically. To do this, we prepare two lemmas.

    \begin{lemma}\label{connected sum}
        For any $\bb \in \Z^n$, $M_{1,\bb}$ is homeomorphic to either $\CP^{n+1} \# \overline{\CP^{n+1}}$ or $\CP^{n+1} \# \CP^{n+1}$.
    \end{lemma}

    \begin{proof}
        Let $N$ be a quasitoric manifold over an $(n+1)$-dimensional polytope $P$ with the characteristic function $\lambda$. Let $F_1,\ldots,F_{n+1}$ be the facets of $P$ meeting at a vertex $q$ of $P$. Then from the non-singularity condition \eqref{non-singularity} we have
        $$\det(\lambda(F_1),\ldots,\lambda(F_{n+1}))=\pm 1.$$
        Let $\vc(P)$ be the vertex cut of $P$ about the vertex $q$ of $P$, and let $G$ be the new facet of $\vc(P)$ obtained from the vertex cut. Let $F_1,\ldots,F_{n+1}$ still denote the facets surrounding the facet $G$
        as in
        Figure~\ref{fig:vertex cut of a polytope}. If we extend the characteristic function $\lambda$ to the facets of $\vc(P)$, then the corresponding quasitoric manifold over $\vc(P)$ is a connected sum of $N$ with $\CP^{n+1}$ or $\overline{\CP^{n+1}}$.

        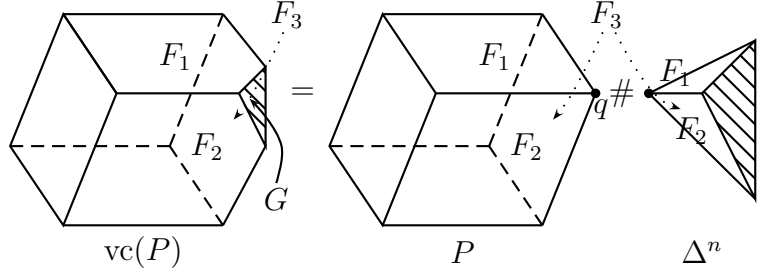
\begin{figure}[h]
        \psset{unit=.7cm}
        \begin{center}
            \begin{pspicture}(-5,-3)(10,3)
                \pspolygon[](-5,-0.5)(-4,-2)(-3,0.5)(-4,2)
                \pspolygon[fillstyle=vlines](-0.2,-0.5)(-0.7,0.5)(-0.2,1)
                \psline[](-3,0.5)(-0.7,0.5)
                \psline[](-0.2,1)(-1,2)
                \psline[](-1,2)(-4,2)
                \psline[](-3,0.5)(-4,2)
                \psline[](-4,-2)(-1,-2)
                \psline[](-1,-2)(-0.2,-0.5)
                \psline[linestyle=dashed](-5,-0.5)(-2,-0.5)
                \psline[linestyle=dashed](-1,2)(-2,-0.5)
                \psline[linestyle=dashed](-2,-0.5)(-1,-2)
                \rput(-1.9,1.2){$F_1$}
                \rput(-1.3,-0.5){$F_2$}
                \rput(0.2,2){$F_3$}
                \pscurve[linestyle=dotted]{->}(0.2,1.8)(-0.5,0.4)(-0.8,0)
                \pscurve[]{->}(0,-1.2)(0.1,-0.5)(-0.5,0.4)
                \rput(0,-1.5){$G$}
                \rput(-2.5,-2.5){$\mathrm{vc}(P)$}
                \rput(0.5,0.5){$=$}
                \pspolygon[](1,-0.5)(2,-2)(3,0.5)(2,2)
                \psline[](3,0.5)(6,0.5)
                \psline[](2,2)(5,2)
                \psline[](5,2)(6,0.5)
                \psline[](2,-2)(5,-2)
                \psline[](5,-2)(6,0.5)
                \psline[linestyle=dashed](1,-0.5)(4,-0.5)
                \psline[linestyle=dashed](4,-0.5)(5,2)
                \psline[linestyle=dashed](4,-0.5)(5,-2)
                \rput(4.1,1.2){$F_1$}
                \rput(4.7,-0.5){$F_2$}
                \rput(6.2,2){$F_3$}
                \psdot(6,0.5)
                \rput(6.1,0.2){$q$}
                \pscurve[linestyle=dotted]{->}(6.2,1.8)(5.5,0.4)(5.2,0)
                \pscurve[linestyle=dotted]{->}(6.3,1.8)(7.2,0.4)(7.6,0.2)
                \rput(3.5,-2.5){$P$}
                \rput(6.5,0.5){$\#$}
                \psline[](7,0.5)(8,0.5)
                \pspolygon[fillstyle=vlines](8,0.5)(9,1.5)(9,-1.5)
                \psline[](7,0.5)(9,1.5)
                \psline[](7,0.5)(9,-1.5)
                \psdot(7,0.5)
                \rput(7.5,0.9){$F_1$}
                \rput(7.8,-0.2){$F_2$}
                \rput(8,-2.5){$\Delta^n$}
            \end{pspicture}
        \end{center}
        \caption{The vertex cut of a polytope $P$}
        \label{fig:vertex cut of a polytope}
    \end{figure}

        Recall the ordering~\eqref{ordering of facets} of the facets of $\Delta^n \times \Delta^1$. Since $\Delta^n \times \Delta^1$ can be viewed as a vertex cut of $\Delta^{n+1}$, the condition
        $$
        \det(\lambda(F_1 \times \Delta^1),\ldots,\lambda(F_n \times \Delta^1),\lambda(F_{n+1} \times \Delta^1))= -a = -1
        $$
        implies that the characteristic function $\lambda$ can be considered as the one extended from a characteristic function on $\Delta^{n+1}$. Therefore $M_{1,\bb}$ is homeomorphic to either $\CP^{n+1}\#\overline{\CP^{n+1}}$ or $\CP^{n+1}\# \CP^{n+1}$.
    \end{proof}

    As we have seen in Remark~\ref{free action}, the moment angle manifold $\mathcal{Z}_{\Delta^n \times \Delta^1}$ is
    $$S^{2n+1}\times S^3=\{(\vw,\vz) \in \C^{n+1} \times \C^2 \colon |\vw|=1,~|\vz|=1\}$$
    and the subtorus $\ker(\overline{\lambda}) \subset T^{n+3}$ is represented by the unimodular subgroup of $\Z^{n+3}$ spanned by $(1,\ldots,1,a,0)$ and $(b_1,\ldots,b_n,0,1,1)$.
    In this section, we denote the subtorus $\ker(\overline{\lambda})$ by $K_{a,\bb}$.

    Assume that we have two quasitoric manifolds $M_{a,\bb}$ and $M_{a^\prime,\bb^\prime}$. If there is a $\theta$-equivariant homeomorphism $\varphi$ from $\mathcal{Z}_{\Delta^n \times \Delta^1}$ with the action of the subgroup $K_{a,\bb} \subset T^{n+3}$ to $\mathcal{Z}_{\Delta^n \times \Delta^1}$ with the action of the subgroup $K_{a^\prime,\bb^\prime} \subset T^{n+3}$, where $\theta$ is an isomorphism from $K_{a,\bb}$ to $K_{a^\prime,\bb^\prime}$, then $\varphi$ induces a homeomorphism $$\overline{\varphi}:M_{a,\bb}=\mathcal{Z}_{\Delta^n \times \Delta^1}/K_{a,\bb} \to M_{a^\prime,\bb^\prime}=\mathcal{Z}_{\Delta^n \times \Delta^1}/K_{a^\prime,\bb^\prime}.$$

    \begin{lemma} \label{lem:class a=pm2}
        Let $n>1$, $\bb=(b,\ldots,b,0,\ldots,0) \in \Z^n$, and $ab=2$. Then we have
        \begin{enumerate}
            \item $M_{a,(b,0,\ldots,0)}$ is homeomorphic to $M_{a,(b,\ldots,b)}$, and
            \item $M_{a,\bb}$ is either homeomorphic to $M_{a,\mathbf{0}}$ if $s$ is even, or $M_{a,(b,0,\ldots,0)}$ if $s$ is odd, where $s$ is the number of $b$'s in $\bb$.
        \end{enumerate}
        In particular, if $n$ is even, then $M_{a,\bb}$ is homeomorphic to $M_{a,\mathbf{0}}$.
    \end{lemma}
    \begin{proof}
    (1)
    Let $\bb = (b,0,\ldots, 0)$ and $\bb' = (b, \ldots, b)$.
    Then, by \eqref{eqn:homomorphism}, there are isomorphisms $\mu\colon T^2 \to K_{a,\bb} \subset T^{n+3}$ and $\mu'\colon T^2 \to K_{a,\bb'} \subset T^{n+3}$ defined by
    \begin{equation*}
        \left(\begin{array}{cc}
        1&b\\1&0\\\vdots&\vdots\\1&0\\a&1\\0&1
        \end{array}\right)
        \mbox{ and }
        \left(\begin{array}{cc}
        1&b\\\vdots&\vdots\\1&b\\1&0\\a&1\\0&1
        \end{array}\right),
    \end{equation*}
    respectively.
    We set $(\vw,\vz) = (w_1, \ldots, w_{n+1}, z_1, z_2) \in S^{2n+1} \times S^3 \subset \C^{n+1} \times \C^2.$
    We define an isomorphism $\theta \colon K_{a,\bb} \to K_{a,\bb'} $ by
    $\mu(t_1,t_2) \mapsto \mu'(t_1t_2^b,t_2^{-1})$
    and a map $\varphi\colon S^{2n+1} \times S^3 \to S^{2n+1} \times S^3$ by
    $$
    \varphi(w_1, \ldots, w_{n+1},z_1,z_2) = (w_{n+1}, w_2, \ldots, w_n, w_1, z_1,\overline{z_2}).
    $$

    Let us check that $\varphi$ is $\theta$-equivariant;
    \begin{align*}
        &\varphi(\mu(t_1,t_2)\cdot(\vw,\vz))\\
        &\,\,=\varphi(t_1t_2^bw_1,t_1w_2,\ldots,t_1w_{n+1},t_1^at_2z_1,t_2z_2)\\
        &\,\,=(t_1w_{n+1},t_1w_2,\ldots,t_1w_n,t_1t_2^bw_1,t_1^at_2z_1,t_2^{-1}\overline{z_2})\\
        &\,\,=(t_1t_2^b(t_2^{-1})^bw_{n+1},\ldots,t_1t_2^b(t_2^{-1})^bw_n,t_1t_2^bw_1,(t_1t_2^b)^at_2^{-1}z_1,t_2^{-1}\overline{z_2})\\
        &\,\,=\mu'(t_1t_2^b,t_2^{-1})\cdot\varphi(\vw,\vz)\\
        &\,\,=\theta(\mu(t_1,t_2))\cdot\varphi(\vw,\vz)
    \end{align*}
    because $ab=2$. Hence $\varphi$ is a $\theta$-equivariant homeomorphism which induces a homeomorphism
    $\overline{\varphi}: M_{a,\bb} \to M_{a,\bb'}$.

    (2)
    By Lemma~\ref{connected sum}, $M_{1,\bb}$ is homeomorphic to $\CP^{n+1}\#\overline{\CP^{n+1}}$ or $\CP^{n+1}\#\CP^{n+1}$. Note that $M_{1,\mathbf{0}}=\CP^{n+1}\#\overline{\CP^{n+1}}$. If $n$ is even, $\CP^{n+1}$ has an orientation-reversing self-homeomorphism. Thus  $\CP^{n+1}\#\overline{\CP^{n+1}}$ is homeomorphic to $\CP^{n+1}\#\CP^{n+1}$. So each $M_{1,\bb}$ is homeomorphic to $M_{1,\mathbf{0}}$. If $n$ is odd, then we have
    \begin{equation*}
        H^\ast(M_{1,\bb})=\left\{\begin{array}{l}H^\ast(\CP^{n+1}\#\overline{\CP^{n+1}})\mbox{ if $s$ is even,}\\
        H^\ast(\CP^{n+1}\#\CP^{n+1})\mbox{ if $s$ is odd.}
        \end{array}\right.
    \end{equation*}
    We note that $H^\ast(M_{1,\mathbf{0}})$ and $H^\ast(M_{1,(2,0,\ldots,0)})$ are not isomorphic as graded rings. (We refer the reader to see the proof of Theorem~\ref{thm:classify NGB for n=1} below.)
     Therefore, $M_{1,\bb}$ is either homeomorphic to $M_{1,\mathbf{0}}=\CP^{n+1}\#\overline{\CP^{n+1}}$ if $s$ is even, or $M_{1,(2,0,\ldots,0)}=\CP^{n+1}\#\CP^{n+1}$ if $s$ is odd.

        Now, consider the case $a=2$.
        Let $\bb = (\underbrace{1,\ldots,1}_{s}, 0,\ldots,0)$, $\bb' = \mathbf{0}$, and $\bb''=(1,0,\ldots,0)$.
        Then, by \eqref{eqn:homomorphism}, there are isomorphisms $\mu: T^2 \to K_{2,\bb} \subset T^{n+3}$, $\mu': T^2 \to K_{2,\bb'} \subset T^{n+3}$, and $\mu'': T^2 \to K_{2,\bb''} \subset T^{n+3}$ defined by
        \begin{equation*}
            \left(\begin{array}{cc}1&1\\\vdots&\vdots\\1&1\\1&0\\\vdots&\vdots\\1&0\\2&1\\0&1\end{array}\right),
            \left(\begin{array}{cc}1&0\\\vdots&\vdots\\1&0\\2&1\\0&1\end{array}\right), \mbox{ and }
            \left(\begin{array}{cc}1&1\\1&0\\\vdots&\vdots\\1&0\\2&1\\0&1\end{array}\right),
        \end{equation*}
        respectively.

        If $s$ is even, we define an isomorphism $\theta\colon K_{2,\bb} \to K_{2,\bb'}$ by
        $\mu(t_1,t_2) \mapsto \mu'(t_1^{-1},t_2^{-1})$
        and a map $\varphi\colon S^{2n+1} \times S^3 \to S^{2n+1} \times S^3$ by
        \begin{equation*}
            \begin{split}
            (\vw,\vz)\mapsto &(\overline{z_1}w_1+z_2\overline{w_2},-z_2\overline{w_1}+\overline{z_1}w_2,\\
            &\qquad\ldots, -z_2\overline{w_{s-1}}+\overline{z_1}w_s,\overline{w_{s+1}},\ldots,\overline{w_{n+1}},\overline{z_1},\overline{z_2}).\\
            \end{split}
        \end{equation*}
       This map is well-defined because $(\overline{z_1}w_{k-1}+z_2\overline{w_k},-z_2\overline{w_{k-1}}+\overline{z_1}w_k)$ comes from the multiplication of quaternion numbers $\overline{z_1+z_2\mathbf{j}}$ and $w_{k-1}+w_{k}\mathbf{j}$ for even $k$ with $2\leq k \leq s$. Then this map $\varphi$ is $\theta$-equivariant because
        \begin{align*}
            &\varphi(\mu(t_1,t_2)\cdot(\vw,\vz))\\
            &=\varphi(t_1t_2w_1,\ldots,t_1t_2w_s,t_1w_{s+1},\ldots,t_1w_{n+1},t_1^2t_2z_1,t_2z_2)\\
            &=(t_1^{-1}(\overline{z_1}w_1+z_2\overline{w_2}),t_1^{-1}(-z_2\overline{w_1}+\overline{z_1}w_2),\\
            &\qquad\ldots, t_1^{-1}(-z_2\overline{w_{s-1}}+\overline{z_1}w_s),t_1^{-1}\overline{w_{s+1}},\ldots,t_1^{-1}\overline{w_{n+1}},
            t_1^{-2}t_2^{-1}\overline{z_1},t_2^{-1}\overline{z_2})\\
            &=\mu'(t_1^{-1},t_2^{-1})\cdot\varphi(\vw,\vz)\\
            &=\theta(\mu(t_1,t_2))\cdot\varphi(\vw,\vz).
        \end{align*}
        Hence $\varphi$ induces a homeomorphism $\overline{\varphi}\colon M_{2,\bb} \to M_{2,\bb'}$

        If $s$ is odd, we define an isomorphism $\theta\colon K_{2,\bb} \to K_{2,\bb''}$ by
        $\mu(t_1,t_2) \mapsto \mu''(t_1^{-1},t_2^{-1})$
        and
        a map $\varphi\colon S^{2n+1} \times S^3 \to S^{2n+1} \times S^3$ by
        \begin{equation*}
            \begin{split}
            (\vw,\vz)\mapsto & (\overline{w_1},\overline{z_1}w_2+z_2\overline{w_3},-z_2\overline{w_2}+\overline{z_1}w_3,\\
            &\qquad\ldots ,-z_2\overline{w_{s-1}}+\overline{z_1}w_{s},\overline{w_{s+1}},\ldots,\overline{w_{n+1}},\overline{z_1},\overline{z_2}).\\
            \end{split}
        \end{equation*} Then this map $\varphi$ is also $\theta$-equivariant because
        \begin{align*}
            &\varphi(\mu(t_1,t_2)\cdot(\vw,\vz))\\
            &=\varphi(t_1t_2w_1,\ldots,t_1t_2w_s,t_1w_{s+1},\ldots,t_1w_{n+1},t_1^2t_2z_1,t_2z_2)\\
            &=(t_1^{-1}t_2^{-1}\overline{w_1},
            t_1^{-1}(\overline{z_1}w_2+z_2\overline{w_3}),t_1^{-1}(-z_2\overline{w_2}+\overline{z_1}w_3),\\
            &\qquad\ldots ,t_1^{-1}(-z_2\overline{w_{s-1}}+\overline{z_1}w_{s}),t_1^{-1}\overline{w_{s+1}},\ldots,t_1^{-1}\overline{w_{n+1}},
            t_1^{-2}t_2^{-1}\overline{z_1},t_2^{-1}\overline{z_2})\\
            &=\mu''(t_1^{-1},t_2^{-1})\cdot\varphi(\vw,\vz)\\
            &=\theta(\mu(t_1,t_2))\cdot\varphi(\vw,\vz).
        \end{align*}
       Hence $\varphi$ induces a homeomorphism $\overline{\varphi}\colon M_{2,\bb} \to M_{2,\bb''}$.
    \end{proof}

    Now, we are ready to prove the following topological classification of quasitoric manifolds over $\Delta^n \times \Delta^1$ which are not projective toric manifolds.

    \begin{theorem}\label{thm:classify NGB for n=1}
    Let $n>1$, $\bb=(b,\ldots,b,0,\ldots,0)\in\Z^n$, and $ab=2$. Then the homeomorphism classes of quasitoric manifolds $M_{a,\bb}$ are represented by the following.
    \begin{enumerate}
    \item $M_{1,\mathbf{0}}$ and $M_{2,\mathbf{0}}$, if $n$ is even, or
    \item $M_{1,\mathbf{0}},~M_{2,\mathbf{0}},~M_{1,(2,0,\ldots,0)}$ and $M_{2,(1,0,\ldots,0)}$, if $n$ is odd.
    \end{enumerate}
       Furthermore, the cohomology ring of each class is distinct.
    \end{theorem}
    \begin{proof}
    By Lemma~\ref{lem:class a=pm2}, each quasitoric manifold over $\Delta^n \times \Delta^1$ is homeomorphic to one of them. Hence, it is enough to show the last statement.

    We note that, by Proposition~\ref{classification of GB with m=1}, the cohomology rings of $M_{1,\mathbf{0}}$ and $M_{2,\mathbf{0}}$ are distinct. Thus, it suffices to show that if $n$ is odd and $a'b'=2$, then $H^\ast(M_{a,\mathbf{0}}) \not\cong H^\ast(M_{a',(b',0,\ldots,0)})$ and $H^\ast(M_{1,(2,0,\ldots,0)}) \not\cong H^\ast(M_{2,(1,0,\ldots,0)})$.

    We denote $M=M_{1,(2,0,\ldots,0)}$ and $N=M_{2,(1,0,\ldots,0)}$. Then
    \begin{align*}
        H^\ast(M) & = \Z[x_1,x_2]/ \langle x_1^n (x_1 + 2x_2), x_2(x_1 + x_2)\rangle \\
        H^\ast(N) & = \Z[y_1,y_2]/ \langle y_1^n (y_1 + y_2), y_2(2y_1 + y_2)\rangle.
    \end{align*}

    We first claim that $H^\ast(M_{a,\mathbf{0}})$ is neither isomorphic to $H^\ast(M)$ nor $H^\ast(N)$ if $n$ is odd and greater than $1$.
    Since $x_1x_2 = -x_2^2$ and $x_1^{n+1} = -2x_2 x_1^n$ in $H^\ast(M)$,
    for any linear element $cx_1+dx_2 \in H^\ast(M)$, we have
    \begin{align*}
        (cx_1+dx_2)^{n+1}&=\sum_{i=0}^{n+1}{{n+1}\choose{i}}(cx_1)^i (dx_2)^{n+1-i}\\
        &=(cx_1)^{n+1}+\sum_{i=0}^{n}(-1)^i{{n+1}\choose{i}}c^id^{n+1-i}x_2^{n+1}\\
        &=2c^{n+1}x_2^{n+1}+\sum_{i=0}^{n}{{n+1}\choose{i}}(-c)^id^{n+1-i}x_2^{n+1}\\
        &=(c^{n+1}+(-c+d)^{n+1})x_2^{n+1}
    \end{align*}
    in $H^\ast(M)$. Since
    $x_2^{n+1}$ does not vanish in $H^\ast(M)$, $(cx_1+dx_2)^{n+1}$ cannot be zero in $H^\ast(M)$ for odd $n>1$.
    Similarly, we can see that
    $$
    (cy_1 + dy_2)^{n+1} = \left(\frac{c^{n+1} + (c-2d)^{n+1}}{2} \right)y_1^{n+1}
    $$
    cannot be zero in $H^\ast(N)$ for odd $n>1$. Since there is a linear element in $H^\ast(M_{a,\mathbf{0}})$ whose $(n+1)$-st power vanishes, neither can $H^\ast(M_{a,\mathbf{0}})$ be isomorphic to $H^\ast(M)$ nor $H^\ast(N)$ for odd $n>1$.

    We finally claim that $H^\ast(M)$ is not isomorphic to $H^\ast(N)$.
    Suppose that there is a grading preserving isomorphism $$\phi \colon H^\ast(M)=\Z[x_1,x_2]/\mathcal{I}_M \to H^\ast(N)=\Z[y_1,y_2]/\mathcal{I}_N$$ which lifts to a grading preserving isomorphism $\bar{\phi}: \Z[x_1,x_2] \to \Z[y_1,y_2]$ with $\bar{\phi}(\mathcal{I}_M)=\mathcal{I}_N$. Since $\bar{\phi}(\mathcal{I}_M)=\mathcal{I}_N$ and $n>1$, we have
    \begin{equation}\label{add_1 added added}
        \bar{\phi}(x_2(x_1+x_2))=\alpha y_2(2y_1+y_2),
    \end{equation}
    where $\alpha$ is a nonzero integer. The prime divisors of the left hand side of~\eqref{add_1 added added} generate $\Z[x_1,x_2]$ as a $\Z$-algebra, whereas the prime divisors of the right hand side of~\eqref{add_1 added added} do not generate $\Z[y_1,y_2]$. Therefore, $H^\ast(M)$ and $H^\ast(N)$ cannot be isomorphic.
    \end{proof}

    \begin{corollary} \label{cohomological rigidity 1}
        Two quasitoric manifolds over $\Delta^n \times \Delta^1$ are homeomorphic if their cohomology rings are isomorphic as graded rings. In particular,
        \begin{enumerate}
            \item if $n$ is even, then $M$ is homeomorphic to a generalized Bott manifold $M_{a,\mathbf{0}}$ or $M_{0,\mathbf{b}}$, and
            \item if $n$ is odd, then $M$ is homeomorphic to a generalized Bott manifold or $M_{1,(2,0,\ldots,0)} \cong \CP^{n+1}\# \CP^{n+1}$ or $M_{2,(1,0,\ldots,0)}$.
        \end{enumerate}
    \end{corollary}
    \begin{proof}
    Let $M$ and $N$ be quasitoric manifolds over $\Delta^n \times \Delta^1$. Assume that $H^\ast(M) \cong H^\ast(N)$. When $n=1$, $M$ is homeomorphic to $N$ by Example~\ref{exam:hirz}.

    Now consider the case when $n>1$.
    If $M$ is equivalent to a generalized Bott manifold $M_{0,\bb}$, then $N$ is also equivalent to a generalized Bott manifold by Proposition~\ref{prop:GB_int}, so $M$ and $N$ are homeomorphic by Theorem~\ref{thm:classification of CMS}.

    If $M$ is equivalent to a generalized Bott manifold $M_{a,\textbf{0}}$, then
    $N:=M_{a',\bb'}$ must be homeomorphic to a generalized Bott manifold $M_{a',\textbf{0}}$
    because $H^\ast(M_{a,\mathbf{0}})$ cannot be isomorphic to $H^\ast(M_{a',(b',0,\ldots,0)})$ as in the proof of Theorem~\ref{thm:classify NGB for n=1}.
    Therefore $M$ and $N$ are homeomorphic
    by Theorem~\ref{thm:classification of CMS}.

    If neither $M$ nor $N$ is equivalent to a generalized Bott manifold, then the assertion is true by Theorem~\ref{thm:classify NGB for n=1}.

    Hence, for any case, $M$ is homeomorphic to $N$.
    The latter statement of the corollary immediately follows
    Theorem~\ref{thm:classify NGB for n=1}.
    \end{proof}

    The above corollary proves a part of Theorem~\ref{classify}.

\begin{example}
        There are quasitoric manifolds homeomorphic but not equivalent to generalized Bott manifolds. For example, $M_{2,(1,1,0,\ldots,0)}$ is homeomorphic to a generalized Bott manifold $M_{2,(0,\ldots,0)}$. But $M_{2,(1,1,0,\ldots,0)}$ is not equivalent to a generalized Bott manifold by Proposition~\ref{prop:condition to be GBM}.
\end{example}

\section{Quasitoric manifolds over $\Delta^n \times \Delta^m$ with $n,m>1$} \label{sec:quasi mfds over D^m x D^n}

    As is defined in Section~\ref{sec:quasitoric manifolds of $b_2=2$}, let $M_{\va,\bb}$ be a quasitoric manifold over $\Delta^n \times \Delta^m$ with $n,m >1$ whose characteristic matrix is of the form \eqref{char matrix by reordering}. Define two vectors $\vs$ and $\vr$ by
    \begin{equation}\label{vectors s and r}
    \vs:=(\underbrace{2,\ldots,2}_s,0,\ldots,0) \in \Z^m \text{ and } \vr:=(\underbrace{1,\ldots,1}_r,0,\ldots,0) \in \Z^n,
    \end{equation}
    where $1 \leq s \leq m$ and $1 \leq r \leq n$. If a quasitoric manifold $M$ with $\betti_2=2$ is not equivalent to a generalized Bott manifold, then $M$ is equivalent to $M_{\vs,\vr}$ for some $\vs$ and $\vr$.

    In this section we prove Theorem~\ref{classify} and Theorem~\ref{main} when $n,m>1$. In doing so, we follow the same strategy as the one used in Section~\ref{sec:classification of quasitoric manifolds when m=1}. Assume that we have two quasitoric manifolds $M_{\va,\bb}$ and $M_{\va',\bb'}$. If there is a $\theta$-equivariant homeomorphism $\varphi$ from $\mathcal{Z}_{\Delta^n \times \Delta^m}$ with the subtorus $K_{\va,\bb} \subset T^{n+m+2}$ action to $\mathcal{Z}_{\Delta^n \times \Delta^m}$ with the subtorus $K_{\va',\bb'} \subset T^{n+m+2}$ action, where $\theta$ is an isomorphism from $K_{\va,\bb}$ to $K_{\va',\bb'}$, then $\varphi$ induces a homeomorphism $$\overline{\varphi}:M_{\va,\bb'}=\mathcal{Z}_{\Delta^n \times \Delta^m}/K_{\va,\bb} \rightarrow M_{\va',\bb'}=\mathcal{Z}_{\Delta^n \times \Delta^m}/K_{\va',\bb'}.$$

    \begin{lemma}\label{lem:NGB}
        Two quasitoric manifolds $M_{\vs,\vr}$ and $M_{\vs',\vr'}$ are homeomorphic if the two pairs $(\vs,\vr)$ and $(\vs',\vr')$ satisfy
        \begin{center}
            $s=s'$ or $s+s'=m+1$, and\\
            $r=r'$ or $r+r'=n+1$,
        \end{center}
        where $\vs,\vs' \in \Z^m$ and $\vr,\vr' \in \Z^n$ are vectors as in \eqref{vectors s and r}.
    \end{lemma}

    \begin{proof}
        As we have seen in Remark~\ref{free action}, the moment angle manifold $\mathcal{Z}_{\Delta^n \times \Delta^m}$ is
        $$S^{2n+1}\times S^{2m+1}=\{(\vw,\vz) \in \C^{n+1} \times \C^{m+1} \colon |\vw|=1,~|\vz|=1\},$$
        and the subtorus $K_{\vs,\vr}$ in $T^{n+m+2}$ is represented by the unimodular subgroup of $\Z^{n+m+2}$ spanned by $$\uu_s:=(\underbrace{1,\ldots,1}_{n+1},\underbrace{2,\ldots,2}_s,0,\ldots,0) \mbox{ and } \vv_r:=(\underbrace{1,\ldots,1}_r,0,\ldots,0,\underbrace{1,\ldots,1}_{m+1}).$$
        That is, there is an isomorphism $\mu: T^2 \to K_{\vs,\vr}$ defined by the matrix $\left(\begin{array}{cc}\uu_s^T &\vv_r^T \end{array}\right)$.

        First consider the case when $\vs=\vs'$, $r \leq \lfloor \frac{n+1}{2} \rfloor$, and $r'=n+1-r$.
        Then we have an isomorphism $\mu': T^2 \to K_{\vs',\vr'}$ defined by the matrix $\left(\begin{array}{cc}\uu_s^T & \vv_{n+1-r}^T\end{array}\right)$.

        We set $(\vw,\vz) = (w_1, \ldots, w_{n+1}, z_1, \ldots, z_{m+1}) \in S^{2n+1} \times S^{m+1} \subset \C^{n+1} \times \C^{m+1}.$
        Now we define an isomorphism $\theta\colon K_{\vs,\vr} \to K_{\vs',\vr'}$ by
        $\mu(t_1,t_2) \mapsto \mu'(t_1t_2,t_2^{-1})$
        and
        a map $\varphi: S^{2n+1} \times S^{2m+1} \rightarrow S^{2n+1} \times S^{2m+1}$ by
        \begin{equation*}
            \begin{split}
            &\varphi(w_1,\ldots,w_{n+1},z_1,\ldots,z_{m+1})\\
            &\qquad=(w_{r+1},\ldots,w_{n+1},w_1,\ldots,w_r,z_1,\ldots,z_s,\overline{z_{s+1}},\ldots,\overline{z_{m+1}}).\\
            \end{split}
        \end{equation*}

        Let us check that $\varphi$ is $\theta$-equivariant;
        \begin{align*}
        &\varphi(\mu(t_1,t_2)\cdot(\vw,\vz))\\
        &=\varphi(t_1t_2w_1,\ldots,t_1t_2w_r,t_1w_{r+1},\ldots,t_1w_{n+1},\\
        &\qquad\qquad\qquad t_1^2t_2z_1,\ldots,t_1^2t_2z_s,t_2z_{s+1},\ldots,t_2z_{m+1})\\
        &=(t_1w_{r+1},\ldots,t_1w_{n+1},t_1t_2w_1,\ldots,t_1t_2w_r,\\
        &\qquad\qquad\qquad t_1^2t_2z_1,\ldots,t_1^2t_2z_s,t_2^{-1}\overline{z_{s+1}},\ldots,t_2^{-1}\overline{z_{m+1}})\\
        &=\mu'(t_1t_2,t_2^{-1})\cdot\varphi(\vw,\vz)\\
        &=\theta(\mu(t_1,t_2))\cdot\varphi(\vw,\vz)
        \end{align*}
        Hence $\varphi$ induces a homeomorphism $\overline{\varphi}$ from $M_{\vs,\vr}$ to $M_{\vs',\vr'}$.

        We now consider the case when $s \leq \lfloor \frac{m+1}{2} \rfloor$, $s'=m+1-s$, and $\vr=\vr'$.
        Then we have an isomorphism $\mu'': T^2 \to K_{\vs',\vr'}$ defined by the matrix $\left(\begin{array}{cc} \uu_{m+1-2}^T & \vv_r^T \end{array}\right)$.

        We define an isomorphism $\theta\colon K_{\vs,\vr} \to K_{\vs',\vr'}$ by
        $\mu(t_1,t_2) \mapsto \mu''(t_1^{-1},t_1^2t_2)$
        and
        a map
        \begin{equation*}
            \begin{split}
            &
            \varphi(w_1,\ldots,w_{n+1},z_1,\ldots,z_{m+1})\\
            &\qquad=(w_1,\ldots,w_r,\overline{w_{r+1}},\ldots,\overline{w_{n+1}},z_{s+1},\ldots,z_{m+1},z_1,\ldots,z_s).\\
            \end{split}
        \end{equation*}

        Then,
        \begin{align*}
        &\varphi(\mu(t_1,t_2)\cdot(\vw,\vz))\\
        &=(t_1t_2w_1,\ldots,t_1t_2w_r,t_1^{-1}\overline{w_{r+1}},\ldots,t_1^{-1}\overline{w_{n+1}},\\
        &\qquad\qquad\qquad t_2z_{s+1},\ldots,t_2z_{m+1},t_1^2t_2z_1,\ldots,t_1^2t_2z_s)\\
        &=\mu''(t_1^{-1},t_1^2t_2)\cdot\varphi(\vw,\vz)\\
        &=\theta(\mu(t_1,t_2))\cdot\varphi(\vw,\vz).
        \end{align*}
        Thus $\varphi$ is a $\theta$-equivariant homeomorphism which induces a homeomorphism $\overline{\varphi}$ from $M_{\vs,\vr}$ to $M_{\vs',\vr'}$.

        Finally, we note that the case when $r = n+1 - r'$ and $s = m+1 - s'$ immediately follows from the composition of the above two cases.
    \end{proof}

    \begin{theorem}\label{Cohomological rigidity 2}
        Let $M_{\vs,\vr}$ and $M_{\vs',\vr'}$ be quasitoric manifolds as defined above. Then the following are equivalent.
        \begin{enumerate}
            \item $s=s'$ or $s+s'=m+1$, and $r=r'$ or $r+r'=n+1$.
            \item $H^\ast(M_{\vs,\vr})$ and $H^\ast(M_{\vs',\vr'})$ are isomorphic.
            \item $M_{\vs,\vr}$ and $M_{\vs',\vr'}$ are homeomorphic.
        \end{enumerate}
    \end{theorem}
    \begin{proof}
        By Lemma~\ref{lem:NGB}, it suffices to prove the implication (2) $\Rightarrow$ (1).
        Let $\mathcal{I}\subset\Z[x_1,x_2]$ be the ideal generated by the homogeneous polynomials
        $x_1^{n+1-r}(x_1+x_2)^{r}$ and $x_2^{m+1-s}(2x_1+x_2)^{s}$, and let $\mathcal{J}\subset\Z[y_1,y_2]$ be also the ideal generated by $y_1^{n+1-r'}(y_1+y_2)^{r'}$ and $y_2^{m+1-s'}(2y_1+y_2)^{s'}$. Then we have
        $$H^\ast(M_{\vs,\vr})=\Z[x_1,x_2]/\mathcal{I} \mbox{ and } H^\ast(M_{\vs',\vr'})=\Z[y_1,y_2]/\mathcal{J}.$$
        Then the cohomology ring isomorphism $\phi: H^\ast(M_{\vs,\vr}) \to H^\ast(M_{\vs',\vr'})$ lifts to a grading preserving isomorphism $\bar{\phi}: \Z[x_1,x_2] \to \Z[y_1,y_2]$ with $\bar{\phi}(\mathcal{I})=\mathcal{J}$. We divide the proof into three cases: (1) $n>m$, (2) $n<m$, and (3) $n=m$.\\

        \textbf{CASE 1: $n>m$}

        Since $\bar{\phi}(x_2^{m+1-s}(2x_1+x_2)^{s}) \in \mathcal{J}$ and $n>m$, we have
        \begin{equation}\label{equation of NGB when n>m added added}
            \bar{\phi}(x_2^{m+1-s}(2x_1+x_2)^{s})=\alpha y_2^{m+1-s'}(2y_1+y_2)^{s'}
        \end{equation}
        for some nonzero integer $\alpha$. Comparing the multiplicities of the prime divisors of both sides of~\eqref{equation of NGB when n>m added added}, we can easily see that
        $s=s'$ or $s=m+1-s'$. Thus $\bar{\phi}(x_2)$ is either $\pm y_2$ or $\pm(2y_1+y_2)$. Then we obtain the following four cases:
        when $s=s'$,
$$
        \left\{
          \begin{array}{ll}
            \bar{\phi}(x_1) = \mp(y_1+y_2) \quad \text{ and } \quad \bar{\phi}(x_2) = \pm y_2   , & \hbox{\textbf{(i)} } \\
            \bar{\phi}(x_1) = \pm y_1 \quad \text{ and } \quad \bar{\phi}(x_2) = \pm y_2   , & \hbox{\textbf{(ii)} } \\
          \end{array}
        \right.
$$ and when $s+s' = m+1$,
$$
        \left\{
          \begin{array}{ll}
            \bar{\phi}(x_1) = \mp(y_1+y_2) \quad \text{ and } \quad \bar{\phi}(x_2) = \pm(2y_1+y_2)   , & \hbox{\textbf{(iii)}} \\
            \bar{\phi}(x_1) = \mp y_1 \quad \text{ and } \quad \bar{\phi}(x_2) =\pm(2y_1+y_2)  . & \hbox{\textbf{(iv)} } \\
          \end{array}
        \right.
$$
        One can check that the cases \textbf{(i)} and \textbf{(iii)} imply that $r+r'=n+1$ and the cases \textbf{(ii)} and \textbf{(iv)} imply that $r=r'$, which proves the implication (2) $\Rightarrow$ (1) in this case.\\

        \textbf{CASE 2: $n<m$}

        This case is quite analogous to the case 1. So we can skip the proof.\\

        \textbf{CASE 3: $n=m$}

        Since $\bar{\phi}(\mathcal{I})=\mathcal{J}$, we have
        \begin{equation}\label{eqn in the system of NGB when n=m added}
        \begin{split}
        \bar{\phi}(x_1^{n+1-r}(x_1+x_2)^r)&=\alpha y_1^{n+1-r'}(y_1+y_2)^{r'}+\alpha^\prime y_2^{n+1-s'}(2y_1+y_2)^{s'},\\
        \bar{\phi}(x_2^{n+1-s}(2x_1+x_2)^s)&=\beta y_1^{n+1-r'}(y_1+y_2)^{r'}+\beta^\prime y_2^{n+1-s'}(2y_1+y_2)^{s'},
        \end{split}
        \end{equation}
        where $\alpha$, $\alpha'$, $\beta$, and $\beta'$ are integers. Note that either $\alpha$ or $\alpha'$ is nonzero, and either $\beta$ or $\beta'$ is nonzero. We first show that $\alpha'$ and $\beta$ are zero and then prove the theorem in this case.

        Plugging $\bar{\phi}(x_i)=g_{i1}y_1+g_{i2}y_2$, $i=1,2$, into~\eqref{eqn in the system of NGB when n=m added}, we have
        \begin{equation}\label{eqn 1 in the system of NGB when n=m added}
            \begin{split}
            &(g_{11}y_1+g_{12}y_2)^{n+1-r}((g_{11}+g_{21})y_1+(g_{12}+g_{22})y_2)^{r}\\
            &\qquad=\alpha y_1^{n+1-r'}(y_1+y_2)^{r'}+\alpha^\prime y_2^{n+1-s'}(2y_1+y_2)^{s'}
            \end{split}
        \end{equation}
        and
        \begin{equation}\label{eqn 2 in the system of NGB when n=m added}
            \begin{split}
            &(g_{21}y_1+g_{22}y_2)^{n+1-s}((2g_{11}+g_{21})y_1+(2g_{12}+g_{22})y_2)^{s}\\
            &\qquad=\beta y_1^{n+1-r'}(y_1+y_2)^{r'}+\beta^\prime y_2^{n+1-s'}(2y_1+y_2)^{s'},
            \end{split}
        \end{equation}
        where the determinant of $G=\left(\begin{array}{cc}g_{11}&g_{12}\\g_{21}&g_{22}\end{array}\right)$ is $\pm 1$.

        Suppose that none of $\alpha$, $\alpha'$, $\beta$, and $\beta'$ are zero. Then by comparing the coefficients of $y_1^{n+1}$ and $y_2^{n+1}$ on both sides of~\eqref{eqn 1 in the system of NGB when n=m added}, we have $\alpha=g_{11}^{n+1-r}(g_{11}+g_{21})^{r}$ and $\alpha'=g_{12}^{n+1-r}(g_{12}+g_{22})^{r}$. By comparing the coefficients of $y_1^{n+1}$ and $y_2^{n+1}$ on both sides of~\eqref{eqn 2 in the system of NGB when n=m added}, we have $\beta=g_{21}^{n+1-s}(2g_{11}+g_{21})^{s}$ and $\beta'=g_{22}^{n+1-s}(2g_{12}+g_{22})^{s}$. Hence we have
         the following system of polynomial equations
        \begin{equation}\label{eqn 1 in the system of NGB when n=m}
            \begin{split}
            &(g_{11}y_1+g_{12}y_2)^{n+1-r}((g_{11}+g_{21})y_1+(g_{12}+g_{22})y_2)^{r}\\
            &\qquad=g_{11}^{n+1-r}(g_{11}+g_{21})^{r}y_1^{n+1-r'}(y_1+y_2)^{r'}\\
            &\qquad\qquad\qquad\qquad + g_{12}^{n+1-r}(g_{12}+g_{22})^{r}y_2^{n+1-s'}(2y_1+y_2)^{s'}
            \end{split}
        \end{equation}
        and
        \begin{equation}\label{eqn 2 in the system of NGB when n=m}
            \begin{split}
            &(g_{21}y_1+g_{22}y_2)^{n+1-s}((2g_{11}+g_{21})y_1+(2g_{12}+g_{22})y_2)^{s}\\
            &\qquad=g_{21}^{n+1-s}(2g_{11}+g_{21})^{s}y_1^{n+1-r'}(y_1+y_2)^{r'}\\
            &\qquad\qquad\qquad\qquad + g_{22}^{n+1-s}(2g_{12}+g_{22})^{s}y_2^{n+1-s'}(2y_1+y_2)^{s'}.
            \end{split}
        \end{equation}

        We first show that $\alpha'=0$.
        Plug $y_1=1$ and $y_2=-1$ into the equation~\eqref{eqn 1 in the system of NGB when n=m} to get the equation
        \begin{equation}\label{1st equation of NGB when n=m}
            \begin{split}
            &(g_{11}-g_{12})^{n+1-r}((g_{11}+g_{21})-(g_{12}+g_{22}))^{r}\\
            &~= g_{12}^{n+1-r}(g_{12}+g_{22})^{r}(-1)^{n+1-s'}.
            \end{split}
        \end{equation}
        Since we assume that $\alpha'$ is not zero, $g_{12}(g_{12}+g_{22}) \neq 0$. Then, by~\eqref{1st equation of NGB when n=m}, we have
        $$
            \left(\frac{g_{11}}{g_{12}}-1\right)^{n+1-r}\left(\frac{g_{11}+g_{21}}{g_{12}+g_{22}}-1\right)^{r}=(-1)^{n+1-s'}.
        $$
        Thus $\frac{g_{11}}{g_{12}}=2$ or $0$, and $\frac{g_{11}+g_{21}}{g_{12}+g_{22}}=2$ or $0$. In these cases, both $g_{11}$ and $g_{21}$ are even, which
        contradicts to $\det(G) = \pm1$. Hence, $\alpha'$ is zero.

        We next show that $\beta=0$.
        Plug $y_1=1$ and $y_2=-2$ into the equation~\eqref{eqn 2 in the system of NGB when n=m} to get the equation
        \begin{equation}\label{2nd equation of NGB when n=m}
            \begin{split}
            &(g_{21}-2g_{22})^{n+1-s}((2g_{11}+g_{21})-2(2g_{12}+g_{22}))^s\\
            &~= g_{21}^{n+1-s}(2g_{11}+g_{21})^s(-1)^{r'}.
            \end{split}
        \end{equation}
        Since we assume that $\beta$ is not zero, $g_{21}(2g_{11}+g_{21})\neq 0$. Then, by~\eqref{2nd equation of NGB when n=m}, we have
        $$
            \left(1-\frac{2g_{22}}{g_{21}}\right)^{n+1-s}\left(1-\frac{2(2g_{12}+g_{22})}{2g_{11}+g_{21}}\right)^s=(-1)^{r'}.
        $$
        Thus $\frac{g_{22}}{g_{21}}=0$ or $1$, and $\frac{2g_{12}+g_{22}}{2g_{11}+g_{21}}=0$ or $1$. In these cases, $\det G \neq \pm 1$ which is a contradiction. Hence, $\beta$ is zero.

        Now we will show that $s=s'$ or $s+s'=m+1$, and $r=r'$ or $r+r'=n+1$. Since both $\alpha'$ and $\beta$ are zero, we have
        \begin{equation*}
            \bar{\phi}(x_1^{n+1-r}(x_1+x_2)^r)=\alpha y_1^{n+1-r'}(y_1+y_2)^{r'}
        \end{equation*}
        and
        \begin{equation*}
            \bar{\phi}(x_2^{n+1-s}(2x_1+x_2)^s)=\beta^\prime y_2^{n+1-s'}(2y_1+y_2)^{s'}.
        \end{equation*}
         Hence, by using the same argument as in case 1, we can show that $s=s'$ or $s+s'=m+1$, and $r=r'$ or $r+r'=n+1$.

    \end{proof}

    \begin{lemma}\label{lem:non-iso cohomology}
        If $n \neq m$, then two quasitoric manifolds $M_{\vs,\vr}$ and $M_{\vr',\vs'}$
        are not homeomorphic for any
        chosen vectors $\vs,\,\vr' \in \Z^m$ and $\vr,\,\vs' \in \Z^n$
        as in~\eqref{vectors s and r}.
        That is,
        \begin{equation*}
        \begin{array}{l}
        \vs:=(\underbrace{2,\ldots,2}_s,0,\ldots,0),~ \vr':=(\underbrace{1,\ldots,1}_{r'},0,\ldots,0)\in \Z^m, \mbox{ and}\\ \vr:=(\underbrace{1,\ldots,1}_r,0,\ldots,0),~ \vs':=(\underbrace{2,\ldots,2}_{s'},0,\ldots,0) \in \Z^n.
        \end{array}
        \end{equation*}
    \end{lemma}
    \begin{proof}
        It is enough to show the case when $n<m$.
        Let $\mathcal{I} \subset \Z[x_1,x_2]$ be the ideal generated by the homogeneous polynomials $x_1^{n+1-r}(x_1+x_2)^r$ and $x_2^{m+1-s}(2x_1+x_2)^s$, and let $\mathcal{J} \subset \Z[y_1,y_2]$ be also the ideal generated by the homogeneous polynomials $y_1^{n+1-s'}(y_1+2y_2)^{s'}$ and $y_2^{m+1-r'}(y_1+y_2)^{r'}$. Then we have
        $$H^\ast(M_{\vs,\vr})=\Z[x_1,x_2]/\langle x_1^{n+1-r}(x_1+x_2)^r,x_2^{m+1-s}(2x_1+x_2)^s \rangle,$$
        and
        $$H^\ast(M_{\vr',\vs'})=\Z[y_1,y_2]/\langle y_1^{n+1-s'}(y_1+2y_2)^{s'},y_2^{m+1-r'}(y_1+y_2)^{r'}\rangle.$$

        Suppose that $M_{\vs,\vr}$ and $M_{\vr',\vs'}$ are homeomorphic for some $\vs,\,\vr' \in \Z^m$ and $\vr,\,\vs' \in \Z^n$.
        Then the ring isomorphism $\phi: H^\ast(M_{\vs,\vr}) \to H^\ast(M_{\vr',\vs'})$ lifts to a grading preserving isomorphism $\bar{\phi}: \Z[x_1,x_2] \to \Z[y_1,y_]$ with $\bar{\phi}(\mathcal{I})=\mathcal{J}$. Then we have $$\bar{\phi}(x_1^{n+1-r}(x_1+x_2)^r)=\alpha y_1^{n+1-s'}(y_1+2y_2)^{s'}$$ for some nonzero integer $\alpha$. But this is a contradiction because the prime divisors of the left hand side generate $\Z[x_1,x_2]$ as a $\Z$-algebra, whereas the prime divisors of the right hand side do not generate $\Z[y_1,y_2]$.

        Therefore, there is no isomorphism between $H^\ast(M_{\vs,\vr})$ and $H^\ast(M_{\vr',\vs'})$, so $M_{\vs,\vr}$ and $M_{\vr',\vs'}$ are not homeomorphic.
    \end{proof}

    \begin{theorem}\label{Cohomological rigidity 3}
        Two quasitoric manifolds over $\Delta^n \times \Delta^m$ with $n,m>1$ are homeomorphic if and only if their cohomology rings are isomorphic as graded rings.
    \end{theorem}
    \begin{proof}
        Let $M$ and $N$ be quasitoric manifolds over $\Delta^n \times \Delta^m$. Assume that $H^\ast(M) \cong H^\ast(N)$.

        If $M$ is equivalent to a generalized Bott manifold, then $N$ is also equivalent to a generalized Bott manifold by Proposition~\ref{prop:GB_int}, so $M$ and $N$ are homeomorphic by Theorem~\ref{thm:classification of CMS}.

        If $M$ is equivalent to $M_{\vs,\vr}$, then $N$ is equivalent to $M_{\vs',\vr'}$ or $M_{\vr',\vs'}$ by Proposition~\ref{prop:GB_int}. But by Lemma~\ref{lem:non-iso cohomology}, $N$ must be equivalent to $M_{\vs',\vr'}$. Thus $M$ and $N$ are homeomorphic by Theorem~\ref{Cohomological rigidity 2}.

        Hence, for any case, $M$ is homeomorphic to $N$.
    \end{proof}

    \begin{corollary}
        Let $N(n,m)$ be the number of quasitoric manifolds over $\Delta^n \times \Delta^m$ which are not homeomorphic to generalized Bott manifolds.
        \begin{enumerate}
            \item When $n=m$, $N(n,n)=\lfloor \frac{n+1}{2}\rfloor \times \lfloor \frac{n+1}{2} \rfloor$.
            \item When $n \neq m$ and $n,m>1$, $N(n,m)=2\lfloor \frac{n+1}{2}\rfloor \times \lfloor \frac{m+1}{2} \rfloor$.
            \item $N(n,1)=0$ for even $n$ and $N(n,1)=2$ for odd $n \geq 3$.
        \end{enumerate}
    \end{corollary}
    \begin{proof}
        It immediately follows from Corollary~\ref{cohomological rigidity 1}, Theorem~\ref{Cohomological rigidity 2}, and Lemma~\ref{lem:non-iso cohomology}.
    \end{proof}

\section{Proof of Theorem~\ref{main}} \label{sec:proof of main}
    A simple polytope $P$ is said to be \emph{cohomologically rigid} if there exists a quasitoric manifold $M$ over $P$, and whenever there exists a quasitoric manifold $N$ over another polytope $Q$ with a graded ring isomorphism $H^{\ast}(M) \cong H^{\ast}(N)$ there is a combinatorial equivalence $P \approx Q$. By \cite{CPS}, a product of simplices is cohomologically rigid.

    Let $M$ and $M'$ be quasitoric manifolds with $\betti_2=2$. Then they are supported by the polytopes combinatorially equivalent to products of two simplices, say $\Delta^n \times \Delta^m$ and $\Delta^{n'} \times \Delta^{m'}$, respectively. Since products of simplices are cohomologically rigid, if $H^\ast(M) = H^\ast(M')$, then $\{n,m\} = \{n', m'\}$. In other words, two quasitoric manifolds over distinct products of simplices can not have the same cohomology rings.

    By Corollary~\ref{cohomological rigidity 1} and Theorem~\ref{Cohomological rigidity 3}, all quasitoric manifolds over a certain product of two simplices are classified by their cohomology rings. Hence, all quasitoric manifolds with $\betti_2=2$ are classified by their cohomology rings as graded rings.

\section{Classification of quasitoric manifolds with $\betti_2=2$}\label{sec:Classification of quasitoric manifolds}

    Let $\uu = (u_1, \ldots, u_k), \uu' = (u'_1, \ldots, u'_k) \in \Z^k$ and let $\ell$ be a positive integer. We define $\uu$ is \emph{equivalent} to $\uu'$ with respect to $\ell$, denote it by $\uu \sim_\ell \uu'$, if there is $\epsilon =\pm1$ and $w \in \Z$ such that
$$
    \prod_{i=1}^k (1+u_ix) = (1+\epsilon wx)\prod_{i=1}^k (1 + \epsilon(u'_i + w)x) \quad \text{ in } \Z[x]/x^{\ell+1}.
$$   Then from Theorem~\ref{thm:classification of CMS}, Example~\ref{exam:hirz},
    Corollary~\ref{cohomological rigidity 1}, Theorem~\ref{Cohomological rigidity 2}, and Theorem~\ref{Cohomological rigidity 3}, we have the following topological classification.

\begin{theorem}
        \begin{enumerate}
            \item The homeomorphism classes of quasitoric manifold over $\Delta^n \times \Delta^m$ with $n \neq m$ ($n,m >1$) are represented by the following:
                \begin{itemize}
                    \item $M_{\mathbf{0}, \mathbf{0}} = \CP^n \times \CP^m$, a trivial generalized Bott manifold.
                    \item $M_{\va,\mathbf{0}}$ for $\va \in (\Z^m-\mathbf{0})/\sim_n$, non-trivial generalized Bott manifolds.
                    \item $M_{\mathbf{0},\bb}$ for $\bb \in (\Z^n-\mathbf{0})/\sim_m$, non-trivial generalized Bott manifolds.
                    \item $M_{\vs,\vr}$ for $\vs:=(2,\ldots,2,0,\ldots,0) \in \Z^m$ and $\vr:=(1,\ldots,1,0,\ldots,0) \in \Z^n$,
                    \item $M_{\vs,\vr}$ for $\vs:=(1,\ldots,1,0,\ldots,0) \in \Z^m$ and $\vr:=(2,\ldots,2,0,\ldots,0) \in \Z^n$,
                \end{itemize}
                where the number of nonzero components in $\vs$, respectively $\vr$, is positive and less than or equal to $\lfloor \frac{m+1}{2} \rfloor$, respectively $\lfloor \frac{n+1}{2} \rfloor$.
            \item The homeomorphism classes of quasitoric manifold over $\Delta^n \times \Delta^n$ ($n >1$) are represented by the following:
                \begin{itemize}
                    \item $M_{\mathbf{0}, \mathbf{0}} = \CP^n \times \CP^n$.
                    \item $M_{\va,\mathbf{0}}$ for $\va \in (\Z^n-\mathbf{0})/\sim_n$.
                    \item $M_{\vs,\vr}$ for $\vs:=(2,\ldots,2,0,\ldots,0) \in \Z^n$ and $\vr:=(1,\ldots,1,0,\ldots,0) \in \Z^n$,
                \end{itemize}
                where the number of nonzero components in $\vs$ and $\vr$ are positive and less than or equal to $\lfloor \frac{n+1}{2} \rfloor$.

            \item The homeomorphism classes of quasitoric manifolds over $\Delta^1 \times \Delta^n$ ($n>1$ is odd) are represented by the following:
                \begin{itemize}
                    \item $M_{0, \mathbf{0}} = \CP^1 \times \CP^n$.
                    \item $M_{a,\mathbf{0}}$ for $a \in \N$.
                    \item $M_{0,\bb}$ for $\bb \in (\Z^n-\mathbf{0}) /\sim_1$ (see Proposition~\ref{classification of GB with m=1}).
                    \item $\CP^{n+1} \# \CP^{n+1}$.
                    \item $M_{2,(1,0,\ldots,0)}$.
                \end{itemize}
            \item The homeomorphism classes of quasitoric manifolds over $\Delta^1 \times \Delta^n$ ($n$ is even) are represented by the following:
                \begin{itemize}
                    \item $M_{0, \mathbf{0}} = \CP^1 \times \CP^n$.
                    \item $M_{a,\mathbf{0}}$ for $a \in \N$.
                    \item $M_{0,\bb}$ for $\bb \in (\Z^n-\mathbf{0}) /\sim_1$ (see Proposition~\ref{classification of GB with m=1}).
                \end{itemize}
            \item The homeomorphism classes of quasitoric manifolds over $\Delta^1 \times \Delta^1$ are represented by the following:
                \begin{itemize}
                    \item $M_{0, 0} = \CP^1 \times \CP^1$.
                    \item $M_{0, 1} = \CP^2 \# \overline{\CP^2}$.
                    \item $M_{2, 1} = \CP^2 \# \CP^2$.
                \end{itemize}
        \end{enumerate}
\end{theorem}

\section*{Acknowledgement}
    The authors are thankful to the referee for kind and careful comments and helpful suggestions which improve the paper.


\bigskip
\begin{thebibliography}{amsplain}

\bibitem[BP02]{BP}
Victor~M. Buchstaber and Taras~E. Panov, \emph{Torus actions and their
  applications in topology and combinatorics}, University Lecture Series,
  vol.~24, American Mathematical Society, Providence, RI, 2002.

\bibitem[CM09]{CM}
Suyoung Choi and Mikiya Masuda, \emph{Classification of $\Q$-trivial Bott manifolds},
to appear J. Symplectic Geom.; arXiv:0912.5000.

\bibitem[CMS11]{ch-ma-su11}
Suyoung Choi, Mikiya Masuda and Dong~Youp Suh, \emph{Rigidity problems in toric topology, a survey}, to appear in Proceedings of the Steklov Institute of Mathematics.; arXiv:1102.1359.

\bibitem[CMS10a]{ch-ma-su08}
Suyoung Choi, Mikiya Masuda and Dong~Youp Suh, \emph{Quasitoric manifolds over
  a product of simplices}, Osaka J. Math. \textbf{47} (2010), no.~1, 109--129.

\bibitem[CMS10b]{CMS}
Suyoung Choi, Mikiya Masuda and Dong~Youp Suh, \emph{Topological classification of generalized {B}ott towers}, Trans.
  Amer. Math. Soc. \textbf{362} (2010), no.~2, 1097 -- 1112.

\bibitem[CPS08]{CPS}
Suyoung Choi, Taras~E. Panov, and Dong~Youp Suh, \emph{Toric cohomological
  rigidity of simple convex polytopes},  J. London Math. Soc. 82(2) (2010), 343--360.

\bibitem[CS09]{CS}
Suyoung Choi and Dong Youp Suh, \emph{Properties of Bott manifolds and cohomological
rigidity}, Algebr. Geom. Topol. 11(2) (2011), 1053--1076.

\bibitem[DJ91]{DJ}
Michael~W. Davis and Tadeusz Januszkiewicz, \emph{Convex polytopes, {C}oxeter
  orbifolds and torus actions}, Duke Math. J. \textbf{62} (1991), no.~2,
  417--451.

\bibitem[Dob01]{D}
Natalia {\`E}. Dobrinskaya, \emph{The classification problem for quasitoric manifolds over a given polytope}, Funktsional. Anal. i Prilozhen. \textbf{35} (2001), no.~2, 3--11, 95.

\bibitem[Gru03]{G}
Branko Grunbaum, \emph{Convex polytopes}, Graduate Texts in Mathematics, 221. Springer-Verlag, New York, 2003. xvi+468 pp.

\bibitem[HY76]{HY}
Akio Hattori and Tomoyoshi Yoshida, \emph{Lifting compact group actions in fiber bundles}, Japan. J. Math. (N.S.) \textbf{2} (1976), no.~1, 13--25.

\bibitem[Hir51]{Hirzebruch-1951}
Friedrich Hirzebruch, \emph{\"{U}ber eine {K}lasse von
  einfachzusammenh\"angenden komplexen {M}annigfaltigkeiten}, Math. Ann.
  \textbf{124} (1951), 77--86.

\bibitem[Hsi66]{H}
Wu-chung Hsiang, \emph{A note on free differentiable actions of {$S^{1}$} and
              {$S^{3}$} on homotopy spheres}, Ann. of Math. (2) \textbf{83} (1966), 266--272.

\bibitem[Kle88]{K}
 Peter Kleinschmidt, \emph{A classification of toric varieties with few generators}, Aequationes Math. \textbf{35} (1988), no.~2-3, 254--266.

\bibitem[MP08]{MP}
Mikiya Masuda and Taras~E. Panov, \emph{Semifree circle actions, {B}ott towers,
  and quasitoric manifolds}, Sbornik Math. \textbf{199} (2008), no.~8,
  1201--1223.

\bibitem[MS]{MS}
Mikiya Masuda and Dong Youp Suh, \emph{Classification problems of toric manifolds via topology}, Contemporary Mathematics \textbf{460} (2008), 273--286.

\end{thebibliography}
\end{document}